\documentclass{article}
\usepackage[utf8]{inputenc}
\usepackage{amsmath}
\usepackage{amssymb}
\usepackage{amsthm}
\usepackage{verbatim}
\usepackage{graphicx}
\usepackage{subcaption}
\usepackage[shortlabels]{enumitem}
\usepackage{xcolor}
\usepackage{float}
\usepackage{lineno}
\usepackage{wasysym}
\usepackage{amsfonts}
\usepackage[normalem]{ulem}

\newtheorem{theorem}{Theorem}[section]

\newtheorem{proposition}[theorem]{Proposition}
\newtheorem{lemma}[theorem]{Lemma} 
\newtheorem{remark}[theorem]{Remark}

\newcommand{\AAA}{{\mathsf{A}}}
\newcommand{\AAK}{\mathsf{A}_k}

\newcommand{\WW}{{\mathsf{W}}}
\newcommand{\UU}{{\mathsf{U}}}
\newcommand{\VV}{{\mathsf{V}}}
\newcommand{\RWk}{{\mathcal{R}_{W_k}}}

\newcommand{\W}{W}
\newcommand{\Wk}{W_k}
\newcommand{\wki}{w_{k,i}}
\newcommand{\wkj}{w_{k,j}}
\newcommand{\wkl}{w_{k,l}}
\newcommand{\Pk}{P_k}
\newcommand{\K}{K}
\newcommand{\Kk}{K_k}
\newcommand{\Kdk}{K_k^\dagger}

\newcommand{\Rn}{{\mathbb{R}^n}}
\newcommand{\RM}{{\mathbb{R}^M}}

\newcommand{\xx}{{\mathbf{x}}}
\newcommand{\yy}{{\mathbf{y}}}

\newcommand{\vv}{{\mathbf{v}}}

\newcommand{\bb}{{\mathbf{b}}}

\newcommand{\uu}{{\mathbf{u}}}

\newcommand{\ee}{{\mathbf{e}}}

\newcommand{\Omegai}{{\Omega_\mathrm{i}}}
\newcommand{\Omegao}{{\Omega_\mathrm{o}}}

\newcommand{\nullspace}[1]{{\mathcal{N}(#1)}}

\newcommand{\rem}[1]{\textcolor{black}{#1}}

\DeclareMathOperator*{\argmin}{arg\,min}
\DeclareMathOperator*{\argmax}{arg\,max}
\DeclareMathOperator*{\Ran}{Ran}

\title{Fictitious null spaces for improving the solution of injective  inverse problems} 
\author{Ole L{\o}seth Elvetun\thanks{Faculty of Science and Technology, Norwegian University of Life Sciences. Email: ole.elvetun@nmbu.no.}, Kim Knudsen\thanks{Department of Applied Mathematics and Computer Science, Technical University of Denmark} and Bj{\o}rn Fredrik Nielsen\thanks{Faculty of Science and Technology, Norwegian University of Life Sciences.}}

\begin{document}

\maketitle

\begin{abstract}
For linear ill-posed  problems with nontrivial null spaces, Tikhonov regularization and truncated singular value decomposition (TSVD) typically yield solutions that are close to the minimum norm solution. Such a bias is not always desirable, and we have therefore in a series of papers developed a weighting procedure which produces solutions with a different and controlled bias. This methodology can also conveniently be invoked when sparsity regularization is employed. 

The purpose of the present work is to study the potential use of this weighting applied to injective operators. The image under a compact operator of the singular vectors/functions associated with very small singular values will be almost zero. Consequently, one may regard these singular vectors/functions to constitute a basis for a fictitious null space that allows us to mimic the previous weighting procedure.  It turns out that this regularization by weighting can improve the solution of injective inverse problems compared with more traditional approaches.


 We present some analysis of this methodology and exemplify it numerically, using sparsity regularization, for three PDE-driven inverse problems: the inverse heat conduction problem, the Cauchy problem for Laplace's equation, and the (linearized) Electrical Impedance Tomography problem with experimental data. 
\end{abstract}

\section{Introduction} \label{sec:introduction}

Consider the equation 
\begin{equation}\label{eq:orig}
    \K x = y,
\end{equation}
where $K: X \rightarrow Y$ is a linear, compact and injective operator between two separable Hilbert spaces $X$ and $Y$ of infinite dimension. If we have noise-free data $y \in \Ran(K) \subset Y,$  then the generating element $x^\dagger \in X$, solving $K x^\dagger = y$, can theoretically be recovered uniquely. 

In applications only some noisy approximation $y^\delta$ of the exact data $y$ is available with an upper bound for 
$\|y - y^\delta\|_Y$, i.e.,   
\begin{equation}\label{eq:noisebound}
\|y - y^\delta\|_Y \leq \delta.
\end{equation}
We can therefore at the best obtain a good approximation of $x^\dagger$, and one often considers the associated variational formulation 
\begin{equation}\label{eq:varform}
    \min_x \left\{\frac{1}{2}\|Kx - y^\delta\|_Y^2 + \alpha \mathcal{R}(x)\right\},
\end{equation}
where $\mathcal{R}: X \rightarrow \mathbb{R}$ represents a regularization functional and $\alpha>0$ is the regularization parameter, which is selected based on the upper bound $\delta$ for the size of the noise. The functional $\mathcal{R}$ is typically independent of $\delta$ and chosen based on which characteristics one desires from the solution, e.g., smoothness, sparsity or jump discontinuities.
One can think of such specific techniques as causing a well-understood and warranted \textit{bias} in the problem. However, for inverse problems where $K$ is non-injective, the regularization operators causing these characteristics are also known to produce other forms of biases, such as the well-known \textit{depth bias} in the inverse EEG problem, see, e.g., \cite{fuchs99}. Such unwanted bias might also be a problem for injective inverse problems due to the presence of very small singular values.



The purpose of this paper is to explore the potential use of a weighting procedure which, in a sense specified later, produce more unbiased solutions (relative to a given basis). This weighting was originally developed for non-injective inverse problems, see \cite{elvetun2021a,elvetun21b,elvetun2023weighted, elvetun2024identifying}.
The weighting methodology for non-injective problems involves the projection of each basis function onto the orthogonal complement of the null space of the forward operator. Hence, if we were to extend this approach in a straightforward manner to a problem with a trivial null space, i.e., an injective operator, this projection would simply become the identity map. As a consequence, we would obtain an unweighted regularization method. Consequently, we need a more sophisticated approach in the injective case to obtain the unbiased solutions we want.

Throughout, we assume that $K$  has a trivial null space. For a severely ill-posed problems with noisy data, only very few singular components can be trusted. This defines a large subspace of $X$ spanned by singular functions corresponding to the small singular values, a subspace that is practically invisible from data. For this subspace we coin the phrase \textit{fictitious null spaces}. 

The main novelty of this work is the definition, analysis and demonstration of the weighted regularization method for injective problems with noisy data. We will, based on the fictitious null space and  proper thresholding for small singular values, define the weighting operator as for the non-injective case. This provides a  regularization method as in \eqref{eq:varform} with the weighting operator acting in the regularization term. 
Our analysis shows that, in contrast to the non-weighted approach, we are 
able to recover solutions made up from one basis function with only a marginal error in its magnitude. We consider this to be a sanity check for the method. In addition, we analyse the method in terms of convergence, as the noise tends to zero, in the spirit of regularization theory. We close the paper by demonstrating that for three complex inverse problems, the inverse heat conduction problem, the Cauchy problem for Laplace's equation, and a linearized Electrical Impedance Tomography (EIT) problem, the weighted regularization improves the reconstructions vastly.

\rem{The weighted regularization method applies to both sparsity and classical Tikhonov regularization. In order to keep the exposition clear, we focus in the theoretical development and the computational examples mostly on sparsity regularization. This is a modern approach to the regularization of inverse problems that has gained enormous interest since the pioneering work of Daubechies, Defrise and De Mol in 2004 \cite{Daubechies2004}. We would also like to mention, among others, important contributions from \cite{bredies08,lorenz08,grasmair10,Burger_2013}. See also the special issue of Inverse Problems \cite{jin2017a} for many interesting directions both theoretically, computationally, and in relation to applications. The systematic weighting strategy for sparsity regularization suggested here draws upon \cite{elvetun2023weighted,elvetun2024identifying}; a more heuristic approach to weighted sparsity regularization  was previously used in the context of partial data EIT \cite{garde2015a, garde2016a}.}

The observation that operators with very small singular values have a numerical, or may we say practical, null space which is significantly larger than their exact kernel is not new. For example, {\em numerical rank/effective rank/pseudorank} is discussed in detail in \cite[Section 3.1]{hansen1998rank}. Nevertheless, our approach for handling this issue is, as far as the authors know, new.

The outline of the paper is as follows: We first introduce the necessary notation in Section \ref{sec:stage}. Section \ref{sec:analysis} discusses the bias caused by standard approaches and contains the main theoretical results: The ability for the weighted regularization to recover, with a minor error of its magnitude, a source made from a single basis function. Moreover, a convergence analysis is established. In Section \ref{sec:numerical} we study three specific problems and exemplify numerically how the approach improves the solutions.

Note that we below, for the sake of simple notation, use the notation $\| \cdot \|$ instead of $\| \cdot \|_X$ or $\| \cdot \|_Y$. The context reveals which of these two Hilbert norms that are in use. 

\section{Setting the stage}\label{sec:stage}
Since $K$ is a compact operator, we can represent $\K x$ in terms of its singular value decomposition (SVD)
\begin{equation}\label{eq:SVD}
    \K x = \sum_{i=1}^\infty \sigma_i(x,v_i)u_i,
\end{equation}
where $(v_j, u_j; \sigma_j)$ is a singular system for $\K$ and $\sigma_1 \geq \sigma_2 \geq \ldots > 0$. 
This also enables the definition of the truncated forward operator $\Kk: X \rightarrow Y$,
\begin{equation} \label{eq:tfwd}
    \Kk x = \sum_{i=1}^k \sigma_i(x,v_i)u_i.
\end{equation}
It is well-known that the regularized solution of \eqref{eq:orig}, using truncated SVD (TSVD), reads
\begin{equation}\label{eq:TSVD}
    \Kdk y^\delta = \sum_{i=1}^k \sigma_i^{-1}(y^\delta,u_i)v_i.
\end{equation}
Let us denote $\eta := y^\delta - y$, i.e., $\eta$ contains the noise. Inserting this into \eqref{eq:TSVD} we get
\begin{equation*}
    \Kdk y^\delta = \sum_{i=1}^k \sigma_i^{-1}(y,u_i)v_i + \sum_{i=1}^k \sigma_i^{-1}(\eta,u_i)v_i,
\end{equation*}
\rem{and we want to choose the truncation parameter $k$ such that 
\begin{equation} \label{eq:select_k}
    \frac{(y,u_k)}{\sigma_k} > \frac{(\eta,u_k)}{\sigma_k} \quad \textnormal{and} \quad \frac{(y,u_{k+1})}{\sigma_{k+1}} \le \frac{(\eta,u_{k+1})}{\sigma_{k+1}}.
\end{equation}
That is, $k$ marks the distinction between which components that mainly carry information about the data and which mainly contain noise. Ideally, the first and second inequalities in \eqref{eq:select_k} hold with significant margins.}


Note that the operator $\Kk$ is not injective and that 
\begin{equation*}
    \Kdk y^\delta \in \textnormal{span} \{v_1,v_2,\ldots,v_k \} = \nullspace{\Kk}^\perp,      
\end{equation*}
where $\nullspace{\Kk}$ denotes the null space of $\Kk$.  That is, $\Kdk y^\delta$ yields the minimum norm solution of 
\begin{equation*}
    \min_x \|\Kk x - y^\delta\|^2 .
\end{equation*}
For well-chosen $k$, depending on the noise,  $\nullspace{\Kk}$ is referred to as the fictitious null space of $K.$

Furthermore, when Tikhonov regularization $\mathcal{R}(x)= \| x \|^2$ is employed, the solution of \eqref{eq:varform} is approximately equal to $\Kdk y^\delta$ for appropriate values of $\alpha$. We conclude that we again get solutions that are (almost) minimum norm biased, which might not be desirable for the application at hand. Standard sparsity regularization also introduces bias, but the analysis is far more involved, see \cite[pages 159--160]{burger2013inverse} and \cite[Appendix A]{elvetun2023weighted} for a discussion of this issue in connection with inverse sources problems. 


\section{Weighted regularization: Analysis and convergence}\label{sec:analysis}

In this section we will go further into the weighted regularization. We will discuss and analyze the nature of the bias in standard sparsity and Tikhonov regularization. Moreover, we will perform a  sanity check demonstrating that for particular simple sources, the weighted regularization will retrieve the exact solution with a marginal error in its magnitude; this in  contrast to the standard approach, see \rem{Proposition \ref{proposition:standard_method_fails} and Proposition \ref{proposition:weighted_method_works} in Section \ref{seq:weighted-sparsity}.}

\subsection{Weighted sparsity regularization}\label{seq:weighted-sparsity} 
Sparsity regularization is typically {\em not} invoked in terms of the right singular functions $\{ v_i \}$ of $K$, but rather one employs a basis $\{ \phi_i \}$ for $X$ which is motivated by modelling features. We assume that $\{ \phi_i \}$ constitute an {\em orthonormal basis} for $X$ throughout this text. 

Obviously, $x=\phi_j$ solves the equation $Kx = K \phi_j$. We will now show that this property is not inherited in an intuitive manner\footnote{or may we say trustworthy manner} when standard sparsity regularization is employed: 
\begin{equation} \label{A1}
    \min_x \left\{ \frac{1}{2} \| Kx - K \phi_j \|^2 + \alpha \sum_i | (x, \phi_i) | \right\}.     
\end{equation}
More specifically, we will show that solving \eqref{A1} will {\em not}, in most cases, yield a solution which reflects that the true source only involves the basis function $\phi_j$, even though $K$ is injective. 
\begin{proposition} \label{proposition:standard_method_fails}
    For $\alpha>0$, $x= \gamma \phi_j$ does not solve $\eqref{A1}$ for any scalar $\gamma>0$ unless 
    \begin{equation} \label{A2}
        |(K \phi_j,K \phi_l)| \leq (K \phi_j,K \phi_j) \quad \forall l. 
    \end{equation}
\end{proposition}
\begin{proof}
    The first order optimality condition yields that a solution $x$ of \eqref{A1} must be such that 
    \begin{equation} \label{A3}
        0 \in K^*Kx-K^*K \phi_j + \alpha \partial \sum_i | (x, \phi_i) |, 
    \end{equation}
    where $\partial$ denotes the subgradient/subdifferential. Here, 
    \begin{equation*}
        \partial \sum_i | (x, \phi_i) | = \sum_i \rho_i (x) \phi_i, 
    \end{equation*}
    where 
    \begin{equation}
        \rho_i (x) = \left\{ 
        \begin{array}{ll}
           1,  & (x, \phi_i) > 0, \\
           -1,  & (x, \phi_i) < 0, \\
           \left[ -1,1 \right], & (x, \phi_i) = 0. 
        \end{array}
        \right. \label{eq:rho}
    \end{equation}
    With $x= \gamma \phi_j$, $\gamma >0$, we get  
    \begin{equation*}
        \rho_i (\gamma \phi_j) = \rho_i = \left\{ 
        \begin{array}{ll}
           1,  & i=j, \\
           \left[ -1,1 \right] & i \neq j,  
        \end{array}
        \right.
    \end{equation*}
    and \eqref{A3} becomes 
    \begin{equation*}
        K^* K \phi_j \in \frac{\alpha}{1-\gamma} \sum_i \rho_i \phi_i. 
    \end{equation*}
    For $\gamma \in (0,1)$ we hence obtain the following requirements that $x= \gamma \phi_j$ must fulfill in order to solve \eqref{A3}: 
    \begin{align}
        \label{A4}
        (K^* K \phi_j, \phi_j) &= \frac{\alpha}{1-\gamma}, \\ 
        \label{A5}
        - \frac{\alpha}{1-\gamma} \leq (K^* K \phi_j,& \phi_l) \leq \frac{\alpha}{1-\gamma} \quad l \neq j, 
    \end{align}
    which yields \eqref{A2}. 
    
    The possibility $\gamma=1$ is excluded since the subgradient in \eqref{A3} does not contain the zero element when $x=\phi_j$.  Note also that \eqref{A4} cannot hold if $\gamma > 1$ because $\alpha > 0$ and the left-hand-side of this equation is non-negative.
\end{proof}

Motivated by Proposition \ref{proposition:standard_method_fails}, we will now consider a weighted version of sparsity regularization. To this end, recall the definition \eqref{eq:TSVD} of $\Kdk$ and 
let us define the weighting operator   
\begin{equation}\label{eq:W}
    \Wk \phi_i = \left\{ 
    \begin{array}{cc}
     \| \Pk \phi_i \| \phi_i   & \mbox{if }  \| \Pk \phi_i \| \geq \tau, \\
      \tau \phi_i  &  \mbox{if }  \| \Pk \phi_i \| < \tau,
    \end{array}
    \right.
\end{equation}
for $i=1,2,\ldots,\infty$. Here, $\tau > 0$ is a threshold value and   
\begin{equation} \label{eq:Pk}
    \Pk:X \rightarrow \nullspace{\Kk}^\perp, \quad \Pk=\Kdk \Kk 
\end{equation}
denotes the orthogonal projection operator which maps onto the orthogonal complement $\nullspace{\Kk}^\perp$ of the null space $\nullspace{\Kk}$ of $\Kk$. We have previously analyzed the mathematical and computational properties of $\Wk$ for non-injective operators, focusing on the Euclidean framework and finite dimensional vector spaces without a threshold $\tau$, in a series of papers \cite{elvetun2021a,elvetun21b,elvetun2023weighted,elvetun2024identifying}. 

The need for a threshold in the infinite dimensional setting is motivated by the following observation. 
\begin{proposition}
    Let $\Pk$ be the orthogonal projection \eqref{eq:Pk} and let $\{ \phi_i \}$ be an orthonormal basis for the separable Hilbert space $X$. Then
    \begin{equation*}
        \lim_{i \rightarrow \infty} \| \Pk \phi_i \| = 0. 
    \end{equation*}
\end{proposition}
\begin{proof}
    It follows from \eqref{eq:tfwd} that 
    \begin{equation} \label{eq:Nk_orthogonal}
        \nullspace{\Kk}^\perp = \mathrm{span} \{ v_1, v_2, \ldots, v_k \},  
    \end{equation}
    where $v_1, v_2, \ldots, v_k$ is an orthonormal set of vectors. Now, since $\{ \phi_i \}$ is an orthonormal basis for $X$,  
    \begin{equation*}
        v_j = \sum_{i=1}^\infty (v_j,\phi_i) \phi_i \Rightarrow  \| v_j \|^2 = \sum_{i=1}^\infty
        |(v_j,\phi_i)|^2  
    \end{equation*}
    for $j=1,2\ldots,k$,  which implies that 
    \begin{equation*}
        \lim_{i \rightarrow \infty} |(v_j,\phi_i)|^2 = 0 \mbox{ for } j=1,2\ldots,k. 
    \end{equation*}
    Furthermore, cf. \eqref{eq:Pk} and \eqref{eq:Nk_orthogonal},  
    \begin{equation*}
        \Pk \phi_i = \sum_{j=1}^k (\phi_i,v_j) v_j \Rightarrow \| \Pk \phi_i \|^2 = \sum_{j=1}^k 
        |(v_j,\phi_i)|^2, 
    \end{equation*}
    and we conclude that 
    \begin{equation*}
        \lim_{i \rightarrow \infty} \| \Pk \phi_i \|^2 = \sum_{j=1}^k \lim_{i \rightarrow \infty} 
        |(v_j,\phi_i)|^2 = 0, 
    \end{equation*}
    which completes the argument. 
\end{proof}
\noindent This proposition reveals that, without a threshold, the operator $\W_k:X \rightarrow X$ would not be continuously invertible using the standard norm induced by the inner product of $X$. This is the reason for introducing a threshold $\tau >0$ in \eqref{eq:W}. 

Let us now consider the weighted counterpart to \eqref{A1}:
\begin{equation} \label{A6}
    \min_x \left\{ \frac{1}{2} \| \Pk x - \Kdk K \phi_j \|^2 + \alpha \sum_i \wki | (x, \phi_i) | \right\},      
\end{equation}
where we use the notation 
\begin{equation} \label{A6.01}
    \wki = \left\{ 
    \begin{array}{cc}
     \| \Pk \phi_i \|   & \mbox{if }  \| \Pk \phi_i \| \geq \tau, \\
      \tau &  \mbox{if }  \| \Pk \phi_i \| < \tau.
    \end{array}
    \right. 
\end{equation}
Note that the regularization operator in \eqref{A6} indirectly depends on the upper bound $\delta$ for the noise  because the truncation parameter $k$ typically is depending on $\delta$, cf. the second paragraph of Section \ref{sec:introduction}.  
Also, observe that we use the projection $\Pk$ in the fidelity term in \eqref{A6} instead of the forward operator $\K$. In fact, we have not been able to establish many of the results presented in \cite{elvetun2023weighted,elvetun2024identifying} when employing the standard fidelity term. We will briefly return to this issue below when we analyze the zero-regularization counterpart to \eqref{A6}. 

\rem{In Proposition \ref{proposition:weighted_method_works}, we}  verify that \eqref{A6}, in contrast to \eqref{A1}, has a solution which equals the true source $\phi_j$ multiplied by a positive scalar and that this solution is unique if the images under $\Kk$ of distinct basis functions are not \rem{parallel.}
\begin{proposition} \label{proposition:weighted_method_works}
    For $\alpha \in \left(0, \| \Pk \phi_j \| \right)$, 
    \begin{equation} \label{A6.1}
    x = \left( 1- \frac{\alpha}{\| \Pk \phi_j \|} \right) \phi_j 
    \end{equation}
    solves \eqref{A6} provided that $\|\Pk\phi_j\| \geq \tau$. Furthermore, if in addition 
    \begin{equation}
    \Kk \phi_l \neq c \Kk \phi_q, \quad l \neq q, c \in \mathbb{R}, \label{eq:nonpar}
\end{equation}
then \eqref{A6.1} is the unique solution of \eqref{A6}. 
\end{proposition}
\begin{proof}
    A straightforward calculation reveals that 
    \begin{equation} \label{A6.9}
        \Kdk K = \Kdk \Kk, 
    \end{equation}
    and it follows that we can write \eqref{A6} in the form 
    \begin{equation*}
    \min_x \left\{ \frac{1}{2} \| \Pk x - \Pk \phi_j \|^2 + \alpha \sum_i \wki | (x, \phi_i) | \right\}.       
\end{equation*}
An argument very similar to the reasoning leading to \eqref{A4}-\eqref{A5} reveals that $x=\gamma \phi_j$ solves  \eqref{A6} if, and only if\footnote{The cost-functional \eqref{A6} is convex and the first order optimality condition is thus both sufficient and necessary.},  
    \begin{align}
        \label{A7}
        (\Pk \phi_j, \phi_j) &= \wkj \frac{\alpha}{1-\gamma}, \\ 
        \label{A8}
        - \wkl \frac{\alpha}{1-\gamma} \leq (\Pk \phi_j,& \phi_l) \leq \wkl \frac{\alpha}{1-\gamma} \quad l \neq j, 
    \end{align}
where we used the fact the $\Pk$ is an orthogonal projection and hence that \rem{$\Pk^2 = \Pk$ and $\Pk^* = \Pk$}. 

Equation \eqref{A7} yields that 
\begin{equation} \label{A8.01}
    \gamma = 1 - \frac{\wkj \alpha}{\| \Pk \phi_j \|^2} = 1 - \frac{\alpha}{\| \Pk \phi_j \|}, 
\end{equation} 
where we have used the definition \eqref{A6.01} of $\wkj$ and the assumption that $\| \Pk \phi_j \| \geq \tau$.
Hence,  
\begin{equation*}
    \wkl \frac{\alpha}{1-\gamma} = \wkl \| \Pk \phi_j \|,
\end{equation*}
and it follows that \eqref{A8} can be written in the form 
\begin{equation} \label{A8.1}
    -\wkl \|\Pk\phi_j\| \leq (\Pk \phi_j, \phi_l) \leq \wkl \|\Pk\phi_j\|.  
\end{equation}
Keeping in mind that $\Pk^* \Pk = \Pk$, we deduce from the Cauchy-Schwarz inequality and \eqref{eq:W} that 
\begin{equation*}
(\Pk \phi_j, \phi_l) = (\Pk \phi_j, \Pk \phi_l) \leq \| \Pk \phi_l \| \| \Pk \phi_j \| \leq \wkl \|\Pk\phi_j\|,    
\end{equation*}
and it thus follows that \eqref{A8.1} holds. 

We conclude that 
\begin{equation*}
    x = \gamma \phi_j = \left( 1 - \frac{\alpha}{\| \Pk \phi_j \|} \right) \phi_j 
\end{equation*}
satisfies both \eqref{A7} and \eqref{A8} and hence solves \eqref{A6}. 


To show uniqueness, we start by evaluating 
\begin{equation*}
    \mathcal{J}_{k,\alpha}(x) = \frac{1}{2}\| \Pk x - \Kdk K \phi_j \|^2 + \alpha \sum_i \wki | (x, \phi_i) |      
\end{equation*}
for $x = \gamma\phi_j$, where $\gamma$ is chosen according to \eqref{A8.01}. We get, also employing \eqref{eq:Pk} and \eqref{A6.9}, 
\begin{equation} \label{A8.11}
    \mathcal{J}_{k,\alpha}(\gamma\phi_j) = \frac{1}{2}\alpha^2 + \alpha\gamma\wkj.
\end{equation}
Let $z \in X$ be arbitrary. We will show that, if $r = z - \gamma\phi_j \neq 0$, then $\mathcal{J}_{k,\alpha}(z) > \mathcal{J}_{k,\alpha}(\gamma\phi_j)$. Indeed, again using the assumption that $\| \Pk \phi_j \| \geq \tau$, see also \eqref{A6.01}, 
\begin{eqnarray*}
    \mathcal{J}_{k,\alpha}(z) &=& \frac{1}{2}\| \Pk (\gamma\phi_j + r) - \Pk \phi_j \|^2 + \alpha \sum_i \wki | (\gamma\phi_j + r, \phi_i) | \\ &=&
    \frac{1}{2}\left\| \Pk r - \frac{\alpha}{\wkj} \Pk \phi_j \right\|^2 + \alpha \sum_i \wki | (\gamma\phi_j + r, \phi_i) | \\ &=&
    \frac{1}{2}\left\| \Pk r\right\|^2 - \frac{\alpha}{\wkj} (\Pk r,\Pk \phi_j)  + \frac{1}{2}\alpha^2 + \alpha \sum_i \wki | (\gamma\phi_j + r, \phi_i) |    
\end{eqnarray*}
Furthermore, if we expand $r$ in terms of the $\{\phi_i\}$ basis, i.e., $r = \sum_i (r,\phi_i)\phi_i$, we get
\begin{eqnarray*}
    \mathcal{J}_{k,\alpha}(z) &=&  \frac{1}{2}\left\| \Pk r\right\|^2 - \frac{\alpha}{\wkj} (r,\Pk \phi_j)  + \frac{1}{2}\alpha^2 + \alpha \sum_i \wki | (\gamma\phi_j + r, \phi_i) | \\ &=&
    \frac{1}{2}\left\| \Pk r\right\|^2 - \frac{\alpha}{\wkj} \sum_{i=1}^\infty(r,\phi_i) (\Pk\phi_i, \Pk \phi_j)  + \frac{1}{2}\alpha^2 + \alpha \sum_i \wki | (\gamma\phi_j + r, \phi_i) | \\ &=&
    \frac{1}{2}\left\| \Pk r\right\|^2 - \frac{\alpha}{\wkj} \sum_{i\neq j}(r,\phi_i) (\Pk\phi_i, \Pk \phi_j) - \frac{\alpha}{\wkj} (r,\phi_j)(\Pk\phi_j,\Pk\phi_j) \\ &+& \frac{1}{2}\alpha^2 + \alpha \sum_i \wki | (\gamma\phi_j + r, \phi_i) | \\
    &\geq&
    \frac{1}{2}\left\| \Pk r\right\|^2 - \frac{\alpha}{\wkj} \sum_{i\neq j}|(r,\phi_i)| \wki \wkj - \alpha\wkj (r,\phi_j) \\ &+& \frac{1}{2}\alpha^2 + \alpha \sum_{i\neq j} \wki | (r, \phi_i) | + \alpha \wkj|\gamma + (r, \phi_j)| \\ &=&
    \frac{1}{2}\left\| \Pk r\right\|^2 - \alpha\wkj (r,\phi_j) + \frac{1}{2}\alpha^2  + \alpha \wkj | \gamma + (r, \phi_j) | \\
    &=&
    \frac{1}{2}\left\| \Pk r\right\|^2 + \frac{1}{2}\alpha^2  + \alpha \wkj \left( | \gamma + (r, \phi_j) | - (r,\phi_j) \right).
\end{eqnarray*}
Thus, since 
\begin{equation*}
    |\gamma + (r,\phi_j)| - (r,\phi_j) \geq \gamma > 0
\end{equation*}
for any value of $(r,\phi_j)$, we derive that 
\begin{equation} \label{A8.12}
  \mathcal{J}_{k,\alpha}(z) \geq \frac{1}{2}{\alpha}^2 + \alpha\gamma\wkj + \frac{1}{2}\|\Pk r\|^2.
\end{equation}

In this argument we employed the Cauchy-Schwarz inequality:
\begin{align}
    \nonumber
   \sum_{i\neq j}(r,\phi_i) (\Pk\phi_i, \Pk \phi_j) & \leq \sum_{i\neq j} |(r,\phi_i)| \, |(\Pk\phi_i, \Pk \phi_j)| \\
   \label{A8.2}
   &\leq \sum_{i\neq j}|(r,\phi_i)| \, \| \Pk\phi_i \| \| \Pk \phi_j \| \\
   \nonumber
   &\leq \sum_{i\neq j} |(r,\phi_i)| \, \wki \wkj,  
\end{align}
where the last inequlity follows from \eqref{A6.01}. Recall that $r=z-\gamma \phi_j$. Hence, if $z \neq \gamma \phi_j$, then there must exist $i^* \neq j$ such that $(r,\phi_{i^*}) \neq 0$. Since $\Pk=\Kk^\dagger \Kk$, assumption \eqref{eq:nonpar} implies that $\Pk \phi_{i^*} \neq c \Pk \phi_j, c \in \mathbb{R}$, which yields a 
strict inequality in \eqref{A8.2}. Hence, we conclude that for any $z \neq \gamma\phi_j$, we have
\begin{equation*}
  \mathcal{J}_{k,\alpha}(z) > \frac{1}{2}{\alpha}^2 + \alpha\gamma\wkj = \mathcal{J}_{k,\alpha}(\gamma\phi_j),
\end{equation*}
cf. \eqref{A8.12} and \eqref{A8.11}. 
%
\end{proof}
The uniqueness part of this proof does not only prove the uniqueness of the solution $\gamma \phi_j$, provided that \eqref{eq:nonpar} holds, but also shows that $\gamma \phi_j$, in general, solves \eqref{A6}. Strictly speaking, the first part of the proof, i.e., the part employing the first order optimality conditions, is only needed to obtain the formula \eqref{A8.01} for $\gamma$. Nevertheless, for the sake of completeness, we decided to also include a thorough analysis of the first order optimality conditions.     

\begin{remark}
 The property expressed in Proposition \ref{proposition:weighted_method_works} is what we refer to as {\em unbiased wrt. to the basis $\{ \phi_i \}$}: Each member $\phi_j$ of this basis can be reconstructed, except for a small change in its magnitude, by solving \eqref{A6}.   
\end{remark}

This ability of the weighted regularization method to recover a single basis function - regardless of the size of truncation parameter $k$ - is in stark contrast to the result for the unweighted case, cf. Proposition \ref{proposition:standard_method_fails}. However, to describe a general dependence between the truncation parameter and the capability of the weighted regularization procedure to recover \textit{several} basis functions is more involved. We can, in general, only note that in a truly ideal setting, with noise-free data and without truncation of the forward operator, i.e., the injective case, we could in theory recover any combination of basis functions. 
Nevertheless, it remains to establish a more general almost-exact-recovery result than Proposition \ref{proposition:weighted_method_works}.

As briefly mentioned above, it would be more in line with standard theory to consider 
\begin{equation} \label{A9}
    \min_x \left\{ \frac{1}{2} \| \Kk x - K \phi_j \|^2 + \alpha \sum_i \wki | (x, \phi_i) | \right\}      
\end{equation}
instead of \eqref{A6}. However, we have not succeeded in adapting the argument for Proposition \ref{proposition:weighted_method_works} to the formulation \eqref{A9}. On the other hand, the zero regularization counterpart to \eqref{A9} \rem{is analyzed in Proposition \ref{prop:bp}}, where we note that $K \phi_j$ may not belong the range of $\Kk$ such that a modified basis pursuit problem is \rem{required.} 
\begin{proposition}\label{prop:bp}
    Let $\Kk$ be the truncated approximation \eqref{eq:tfwd} of $\K$ and let $\{ \wki \}$ denote the weights \eqref{A6.01}. Define the set 
    \begin{equation*}
        S = \argmin_{x \in X} \{ \|\Kk x - \K\phi_j\| \}.
    \end{equation*}
    Then,
    \begin{equation} \label{A10}
        \phi_j \in \argmin_{x \in S} \left\{\sum_i \wki| (x,\phi_i)|\right\}, 
    \end{equation}
     provided that $\|\Pk\phi_j\| \geq \tau$. If in addition \eqref{eq:nonpar} holds, then $\phi_j$ is the only solution of \eqref{A10}.
\end{proposition}
\begin{proof}
    We have that
    \begin{alignat*}{2}
        \|\Kk x - \K \phi_j\|^2 &= \|[\Kk x - \Kk\phi_j] &+& [\Kk - \K]\phi_j)\|^2  \\
        &= \|\Kk x - \Kk\phi_j\|^2 &-& 2(\Kk[x -\phi_j], [\Kk - \K]\phi_j) \\
        & &+&\|(\Kk - \K)\phi_j\|^2. 
    \end{alignat*}
    The ranges of $\Kk$ and $\Kk - \K$ are orthogonal because the left singular functions $\{ \sigma_i u_i \} = \{ K v_i \}$ of $K$ are orthogonal. Hence, the inner product in the second term above is zero. Furthermore, the third term is independent of $x$, and we  conclude that 
    $$S = \{x \in X: \Kk x = \Kk\phi_j\}.$$
    
    Consequently, we can rather consider the problem
    \begin{equation} \label{A11}
        \min_{x \in X} \left\{\sum_i \wki| (x,\phi_i)|\right\} \textnormal{ subject to } \Kk x = \Kk\phi_j 
    \end{equation}
    instead of \eqref{A10}. 
    Now, let $x \in X$ be such that $\Kk x = \Kk \phi_j$, which implies that $\Pk x = \Pk \phi_j$, see \eqref{eq:Pk}. Then, since we assume that $\|\Pk\phi_j\| \geq \tau$, see also \eqref{A6.01}, and $\{ \phi_i \}$ is an orthonormal basis for $X$, 
    \begin{eqnarray*}
        \wkj |(\phi_j,\phi_j)| &=& \|\Pk\phi_j\| \\
        &=& \|\Pk x\| \\
        &=& \left\|\sum_i (x,\phi_i) \Pk \phi_i \right\| \\ 
        &\leq& \sum_i \|\Pk\phi_i\| \, |(x,\phi_i)| \\
        &\leq& \sum_i \wki \, |(x,\phi_i)|.
    \end{eqnarray*}
    Finally, the assumption \eqref{eq:nonpar} guarantees that the first inequality above is strict, which completes the proof.
\end{proof} 

\subsection{Weighted Tikhonov regularization} \label{sec:weightedTikhonov}
Let us now consider standard Tikhonov regularization: 
\begin{equation} \label{D0.1}
    \min_x \left\{ \| \K x - \K x^\dagger \|^2 + \alpha \| x \|^2 \right\},  
\end{equation}
where we again explore the possibility for (partly) recovering a true source $x^\dagger$.
Since we assume that the singular values of $\K$ approach zero rapidly, we can approximate \eqref{D0.1} with  
\begin{equation} \label{D0.11}
    \min_x  \left\{ \| \Kk x - \K x^\dagger \|^2 + \alpha \| x \|^2 \right\}, 
\end{equation}
because $\K x \approx \Kk x$, cf. \eqref{eq:SVD} and \eqref{eq:tfwd}. 

The null space of $\Kk$ will influence the solution $x_{k,\alpha}$ of \eqref{D0.11}. More precisely, standard theory reveals that 
\begin{equation}\label{D0.12}
\lim_{\alpha \rightarrow 0} x_{k,\alpha} = \Kk^\dagger \K x^\dagger  \in \nullspace{\Kk}^\perp.   
\end{equation} 
Consequently, unless $x^\dagger \in \nullspace{\Kk}^\perp$, one cannot expect to obtain an unbiased approximation of $x^\dagger$ by solving \eqref{D0.11}. Roughly speaking, since $\Kk$ is non-injective, \eqref{D0.11} becomes biased -- a solution which is almost in the orthogonal complement of the null space of $\Kk$ is preferred, which might not be adequate. \rem{The next result shows} that this bias of \eqref{D0.11} is inherited by \eqref{D0.1} when $\alpha \gg \sigma_{k+1}^2$. Consequently, the weighting operator \eqref{eq:W} might also beneficially be applied in the case of Tikhonov regularization, using the fictitious null space to define the weighting.     

\begin{proposition} \label{proposition:Tikhonov}
    Let $x_\alpha$ and $x_{k,\alpha}$ denote the solutions of \eqref{D0.1} and \eqref{D0.11}, respectively. Then
    \begin{equation*}
    \| x_\alpha-x_{k,\alpha} \|_X \leq \frac{\sigma_{k+1}^2}{\alpha} \| x^\dagger \|. 
\end{equation*}
\end{proposition}
\rem{We omit the proof because this result is a straightforward consequence of the standard expressions
\begin{align*}
    x_\alpha &= \sum_{l=1}^\infty \frac{\sigma_l}{\sigma_l^2+\alpha} (\K x^\dagger, u_l)^2 v_l, \\
    x_{k,\alpha} &= \sum_{l=1}^k \frac{\sigma_l}{\sigma_l^2+\alpha} (\K x^\dagger, u_l)^2 v_l. 
\end{align*}
}

For severely ill-posed problems, $\sigma_{k+1}^2$ typically becomes less than the default precision employed by software systems even for moderate values of $k$, e.g., for $k<10$. 
It hence follows \rem{from Proposition \ref{proposition:Tikhonov}} that the size of  $\| x_\alpha-x_{k,\alpha} \|$ is neglectable, also when, e.g., $\alpha \in [10^{-10},10^{-5}]$. We can therefore conclude that $x_\alpha$, for small $\alpha > 0$ such that $\alpha \gg \sigma_{k+1}^2$, also almost will belong to $\nullspace{\Kk}^\perp$ -- i.e., the property \eqref{D0.12} will often in practice be "inherited" by the solution of \eqref{D0.1}, which might not be desirable. This motivates the use of weighted Tikhonov regularization:  
\begin{equation} \label{eq:weightedTikhonov}
    \min_x \left\{ \| \Kk x - y^\delta \|^2 + \alpha \| \Wk x \|^2 \right\}, 
\end{equation}
where $\Wk$ is defined in \eqref{eq:W}. 

The approach \eqref{eq:weightedTikhonov} is studied in \cite{elvetun2021a,elvetun21b} in the finite dimensional setting and relies on a particular maximum property associated with the minimum norm solution of $\Kk x = \Kk \phi_j$. As mentioned in connection with Proposition \ref{prop:bp} above, $\K \phi_j$ may not belong to the range of $\Kk$, and Theorem 4.2 in \cite{elvetun2021a} must therefore be reconsidered. In the next result we use the notation 
\begin{equation*}
    [z]_i = (z,\phi_i) \quad \mbox{for } z \in X, i=1,2, \ldots .
\end{equation*}
That is, employing the orthonormal basis $\{ \phi_i \}$ for $X$, 
\begin{equation*}
    z = \sum_i (z,\phi_i) \phi_i = \sum_i [z]_i \phi_i.
\end{equation*}

The following lemma will be used in the next subsection to analyze the convergence of weighted sparsity regularization. \rem{Its proof is rather similar to the argument for Theorem 4.2 in \cite{elvetun2021a} and therefore presented in Appendix \ref{sec:proof_maximum}.}
\begin{lemma}\label{prop:maximum}
    Let $\Kk$ be the truncated approximation \eqref{eq:tfwd} of $\K$ and let $\{ \wki \}$ denote the weights \eqref{A6.01}. Define the set 
    \begin{equation*}
        S = \argmin_{x \in X} \{\|\Kk x - \K\phi_j\|\},
    \end{equation*}
    and denote
    \begin{equation*}
        \bar{x}_k = \argmin_{x \in S} \|x\|^2.
    \end{equation*}
    Then,
    \begin{equation}\label{eq:minnorm}
        j \in \argmax_{i \in \mathbb{N}} [\Wk^{-1} \bar{x}_k]_i = \argmax_{i \in \mathbb{N} } [\Wk^{-1} \Kdk \Kk \phi_j]_i.
    \end{equation}
     provided that $\|\Pk\phi_j\| \geq \tau$. If in addition \eqref{eq:nonpar} holds, then $j$ is the only member of the $\argmax$ set \eqref{eq:minnorm}.
\end{lemma}
\rem{Let} us briefly comment that: \eqref{eq:minnorm} shows that $\Wk^{-1} \bar{x}_k$ attains its maximum for the correct index $j$ since the "true source" equals $\phi_j$. Note that $\bar{x}_k$  is the minimum norm solution of $\min_{x \in X} \{\|\Kk x - \K\phi_j\|\}$.

\rem{This lemma and the results presented in \cite{elvetun2021a} suggest that a weighted version of the sparse recovery method analyzed in \cite{2009_Lu}, which only involves quadratic Tikhonov regularization, could achieve similar reconstructions to those obtained with \eqref{A6}, but at a lower computational cost. Since our main concern is the choice of weighting to employ in a sparsity-promoting regularizer, and not computational costs, we have not pursued this issue further. That is, we have not compared the performance of weighted versions of standard sparsity regularization with weighted modifications of the methodology suggested in \cite{2009_Lu}. Nevertheless, we will present some simulations using \eqref{eq:weightedTikhonov} in Section \ref{sec:numerical}.}

\subsection{Convergence analysis}
The main focus of this paper has been to analyze and "mitigate" the discrepancy between the canonical basis $\{ v_i \}$ for $X$, arising from the singular value decomposition of the forward operator $K$, and a basis $\{ \phi_i \}$ chosen based on its merits of representing the solution according to some criterion, e.g., sparsity. 

We will now focus on a different aspect: In order to tie the infinite-dimensional analysis together, we end this section with a classical investigation of the asymptotic behavior, as the size of the noise  tends to zero, of our weighted regularization approach. We limit our discussion to sparsity regularization. 

Let us introduce the notation
\begin{equation} \label{eq:weighted_l1_norm}
    \RWk(x) = \sum_i \wki |(x, \phi_i)|, 
\end{equation}
and consider a sequence $\{ x_{k,\alpha} \}$ satisfying 
\begin{equation}\label{eq:convS}
    x_{k,\alpha} \in \argmin_x\left\{\frac{1}{2}\left\|\Kk x - y^\delta\right\|^2 + \alpha \RWk(x)\right\},
\end{equation}
where 
\begin{equation} \label{eq:exact_solution}
 K x^\dagger = y.    
\end{equation}
Here, as above, $y^\delta$ is a noisy approximation of $y$. In line with the results presented above, we again consider a problem employing $K_k$ instead of $K$ in the fidelity term.

We introduce the source condition
\begin{equation}\label{eq:sparseSC}
    \exists \lambda_k : \ \Kk^*\lambda_k \in \partial \RWk(x^\dagger),
\end{equation}
which is rather standard with the exception that it involves $\Kk$ instead of $K$. 
Then, by modifying the proof of \cite[Theorem 2]{Burger_2004}, we obtain the following result. 
\begin{proposition}\label{thm:convSparse}
    Let $x_{k,\alpha}$ be given as in \eqref{eq:convS}. Assume that the bound \eqref{eq:noisebound} holds and that the source condition \eqref{eq:sparseSC} is satisfied. Then
    \begin{equation} \label{eq:bregman_bound}
        d_k \leq \frac{(\sigma_{k+1}\|x^\dagger\| + \delta)^2}{2\alpha} + \|\lambda_k\|(\delta + \sigma_{k+1}\|x^\dagger\|) + \frac{\alpha}{2}\|\lambda_k\|^2, 
    \end{equation}
    where $d_k$ denotes the Bregman distance between $x_{k,\alpha}$ and $x^\dagger$ for the weighted functional \eqref{eq:weighted_l1_norm} and the specific subderivative given in \eqref{eq:sparseSC}, i.e.,
    \begin{equation}\label{eq:bregman}
        d_k = \RWk(x_{k,\alpha}) - \RWk(x^\dagger) - \left(\Kk^*\lambda_k, x_{k,\alpha} - x^\dagger \right).
    \end{equation}
\end{proposition}

\begin{proof}
    Since $x_{k,\alpha}$ is a minimizer of \eqref{eq:convS}, we \rem{have} 
    \begin{eqnarray*}
        \frac{1}{2}\left\|\Kk x_{k,\alpha} - y^\delta\right\|^2 + \alpha \RWk(x_{k,\alpha}) &\leq& 
        \frac{1}{2}\left\|\Kk x^\dagger - y^\delta\right\|^2 + \alpha \RWk(x^\dagger) \\ &\leq&
        \frac{1}{2}\left\|(\Kk-K) x^\dagger + K x^\dagger - y^\delta\right\|^2 + \alpha \RWk(x^\dagger) \\ &\leq&
        \frac{1}{2}(\sigma_{k+1}\|x^\dagger\| + \delta)^2 + \alpha \RWk(x^\dagger), 
    \end{eqnarray*}
    where we have used \eqref{eq:SVD}, \eqref{eq:tfwd}, \eqref{eq:exact_solution}, the triangle inequality, \eqref{eq:noisebound} and that the singular values $\{ \sigma_j \}$ are sorted non-increasingly. 
    Subtracting $\alpha\RWk(x^\dagger)$ from both sides \rem{we obtain}  
    \begin{equation*}
        \frac{1}{2}\left\|\Kk x_{k,\alpha} - y^\delta\right\|^2 + \alpha d_k + \alpha\left(\partial\RWk(x^\dagger), x_{k,\alpha} - x^\dagger \right) \leq \frac{1}{2}(\sigma_{k+1}\|x^\dagger\| + \delta)^2,
    \end{equation*}
    where $d_k$ is the Bregman distance
    \begin{align*}
        d_k &= \RWk(x_{k,\alpha}) - \RWk(x^\dagger) - \left(\partial\RWk(x^\dagger), x_{k,\alpha} - x^\dagger \right) \\
        &= \RWk(x_{k,\alpha}) - \RWk(x^\dagger) - \left(\Kk^*\lambda_k, x_{k,\alpha} - x^\dagger \right). 
    \end{align*}
    Here, the last equality follows from \eqref{eq:sparseSC}. 
    From \eqref{eq:exact_solution} we have
    \begin{eqnarray*}
        (\Kk^*\lambda_k, x_{k,\alpha} - x^\dagger) &=&
        (\lambda_k, \Kk x_{k,\alpha} - \Kk x^\dagger) \\ &=&
        (\lambda_k, \Kk x_{k,\alpha} - y^\delta + y^\delta - y + K x^\dagger - \Kk x^\dagger) \\ &=&
        (\lambda_k, \Kk x_{k,\alpha} - y^\delta) + (\lambda_k,y^\delta - y + K x^\dagger - \Kk x^\dagger).
    \end{eqnarray*}
    Since 
    \begin{equation*}
        |(\lambda_k,y^\delta - y + K x^\dagger - \Kk x^\dagger)| \leq  \|\lambda_k\|(\delta + \sigma_{k+1}\|x^\dagger\|), 
    \end{equation*}
    we obtain the bound
    \begin{eqnarray*}
        \frac{1}{2}\left\|\Kk x_{k,\alpha} - y^\delta\right\|^2 &+& \alpha d_k + \alpha (\lambda_k, \Kk x_{k,\alpha} - y^\delta) \\
        &\leq& \frac{1}{2}(\sigma_{k+1}\|x^\dagger\| + \delta)^2 + \alpha\|\lambda_k\|(\delta + \sigma_{k+1}\|x^\dagger\|).
    \end{eqnarray*}
    Finally, by adding $\frac{1}{2}\alpha^2\|\lambda_k\|^2$ to both sides of the inequality and noticing that this completes a square on the left-hand side we get 
    \begin{eqnarray*}
      \left\|\Kk x_{k,\alpha} - y^\delta + \alpha\lambda_k\right\|^2 &+& 2\alpha d_k \\ &\leq& (\sigma_{k+1}\|x^\dagger\| + \delta)^2 + 2\alpha\|\lambda_k\|(\delta + \sigma_{k+1}\|x^\dagger\|) + \alpha^2\|\lambda_k\|^2, 
    \end{eqnarray*}
    which, because the first term on the left-hand-side is non-negative, yields the desired bound for the Bregman distance $d_k$. 
\end{proof}


\rem{The authors of \cite{Burger_2013} show that some variational inequalities can provide convergence rates also for sources that are not completely sparse (but with a fast-decaying nonzero part). However, these results rely on a specific choice of a subgradient. More specifically, for all one-sparse solutions, this subgradient must satisfy a very strong source condition. In contrast, our next result} shows that \eqref{eq:sparseSC} always holds when the true source consists of a single basis function.
\begin{proposition} \label{prop:uniform_bound}
    Assume that $x^\dagger = \phi_j$ and that the basis function $\phi_j$ satisfies $\|\Pk\phi_j\| \geq \tau$. Then there exists $\lambda_k$ such that \eqref{eq:sparseSC} holds. Furthermore, if 
    \begin{equation*}
        \sum_{i=1}^\infty \frac{(v_i, \phi_j)^2}{\sigma_i^2} < \infty, 
    \end{equation*}
    then $\| \lambda_k \|$ is uniformly bounded wrt.\ k. In this case, the Bregman distance $d_k$ in \eqref{eq:bregman_bound} reads 
    \begin{equation*}
        d_k = \RWk(x_{k,\alpha}) - \RWk(\phi_j) - \wkj^{-1} \left(\Pk \phi_j, x_{k,\alpha} - \phi_j \right).
    \end{equation*}
\end{proposition}
Before we prove this proposition, we briefly mention that, if $\| \lambda_k \|$ is uniformly bounded, then the bound for $d_k$ in Proposition \ref{thm:convSparse} holds in the limit $k \rightarrow \infty$ and is of order $\mathcal{O}(\delta)$ for $\alpha \sim \delta$.
\begin{proof}
    First, we \rem{have}
    \begin{equation} \label{eq:subdifferentialFormula}
        \partial\RWk(\phi_j) = \sum_i \wki \rho_i(\phi_j) \phi_i,
    \end{equation}
    where $\rho_i$ is defined in \eqref{eq:rho}.
    
    Next, recall that Lemma \ref{prop:maximum} guarantees that $\Wk^{-1} \Kdk \Kk \phi_j$ has its maximal component occurring for index $j$. More precisely, \begin{equation*} 
    \Wk^{-1} \Kdk \Kk \phi_j = \wkj \sum_i \left(\frac{\Pk\phi_j}{\wkj}, \frac{\Pk\phi_i}{\wki}\right) \phi_i, 
    \end{equation*}
    see \eqref{eq:maximumIndexExpansion}. We can write this equality as, cf. the definition \eqref{eq:W} and \eqref{A6.01} of $\Wk$,
    \begin{equation*}
        \Kdk \Kk \phi_j = \wkj \sum_i \wki \left(\frac{\Pk\phi_j}{\wkj}, \frac{\Pk\phi_i}{\wki}\right) \phi_i.
    \end{equation*}
    Since $\Kk^\dagger = \Kk^*(\Kk\Kk^*)^\dagger$, it follows that
    \begin{equation} \label{eq:lambda_k_works}
        \Kk^*(\wkj^{-1}(\Kk\Kk^*)^\dagger\Kk\phi_j)
        = \wkj^{-1}\Kk^\dagger\Kk\phi_j 
        = \sum_i \wki \left(\frac{\Pk\phi_j}{\wkj}, \frac{\Pk\phi_i}{\wki}\right) \phi_i.
    \end{equation}
    Hence, by choosing 
    \begin{equation} \label{eq:formula_lambda_k}
    \lambda_{k,j} = \wkj^{-1}(\Kk\Kk^*)^\dagger\Kk\phi_j,
    \end{equation} 
    invoking the assumption that $\|\Pk\phi_j\| \geq \tau$, and recalling the definition \eqref{A6.01} of $\{ \wki \}$, we conclude from \eqref{eq:lambda_k_works} and \eqref{eq:subdifferentialFormula} that \eqref{eq:sparseSC} is satisfied.

    To prove the uniform boundedness of $\| \lambda_k \|$ wrt. $k$, we employ standard formulas for singular systems and \eqref{eq:formula_lambda_k} to obtain
    \begin{equation*}
        \|\lambda_{k,j}\|^2 = \wkj^{-2}\sum_{i=1}^k \frac{(v_i,\phi_j)^2}{\sigma_i^2}.
    \end{equation*}
    Recall that $\Pk$ is the orthogonal projection onto $\nullspace{\Kk}^\perp = \textnormal{span} \{v_1,v_2,\ldots,v_k \}$ and that $\{ v_i \}$ is an orthonormal basis for $X$. Since $\wkj=\max \{ \| \Pk \phi_j \|, \tau \}$, it follows, for any fixed $j$, that $\wkj \rightarrow 1$ as $k \rightarrow \infty$ due to the injectivity assumption on $K$. This completes the argument for the uniform boundedness of $\| \lambda_k \|$. 
\end{proof}

\section{Numerical experiments}\label{sec:numerical}
We will present numerical results for three model problems: 
\begin{description}
    \item[I)] The inverse heat conduction problem. 
    \item[II)] A Cauchy problem for Laplace's equation. 
    \item[III)] The linearized EIT problem, using experimental data.  
\end{description}
It turns out that the effect of the weighting procedure is rather moderate for I), but has a significant impact on the results obtained for II) and III).

We represent the finite dimensional approximation of \eqref{eq:orig} in terms of a matrix $\AAA$ and vectors $\xx$ and $\bb$: 
\begin{equation} \label{eq:matrix-vector_version}
    \AAA \xx = \bb. 
\end{equation}
In the problems I) and II) we added noise to $\bb$,  
\begin{equation*}
    \bb + \beta \mathbf{\Xi},  
\end{equation*}
where $\beta$ is a scalar and the vector $\Xi$ contained normally distributed numbers
with zero mean and standard deviation $1$. The ratio 
\begin{equation*}
   \frac{\beta}{\max{\bb} - \min{\bb}} 
\end{equation*}
was used to specify the noise level in the simulations, i.e., the magnitude of the noise was measured in terms of the size of the standard deviation of the involved probability distributions relative to the range of the data $\bb$. We employed Morozov's discrepancy principle to determine the size of the regularization parameter $\alpha$.  

The size of the truncation parameter $k$ was selected according to the criterion \eqref{eq:select_k} for problems I) and II). This was possible because problems I) and II) only involve simulated data, and hence the noise vector $\beta \mathbf{\Xi}$ was known. For problem III), however, the magnitudes of $k$ and $\alpha$ were determined by a manual inspection of the outcome of the computations. That is, the number of singular values to include, as well as the size of the regularization parameter, was based on a visual judgement, choosing the results that yielded the best "appearance". It is a challenging problem to develop good methods for mimicking \eqref{eq:select_k} when  precise information about the noise vector is not available, see, e.g., \cite{hansen2010discrete} for further details.  

In the numerical tests we will focus on sparsity regularization. Let $\AAA,$ $\WW,$ and $b$ be finite dimensional counterparts of $\Kk,$ $\Wk,$ and $y,$  respectively. Then the methods discussed in Section \ref{seq:weighted-sparsity}, cf. \eqref{A1}, \eqref{A6} and \eqref{A9}, read: 
\begin{align}
    \label{method:standard_sparsity}
   &\min_{\xx} \left\{ \| \AAA \xx - \bb \|_2^2 + \alpha \| \xx \|_1 \right\}, \\
   \label{method:weighted_sparsity_modified_fidelity}
   &\min_{\xx} \left\{ \| \AAA^\dagger \AAA \xx - \AAA^\dagger \bb \|_2^2 + \alpha \| \WW \xx \|_1 \right\}, \\
   \label{method:weighted_sparsity_standard_fidelity}
   &\min_{\xx} \left\{ \| \AAA \xx - \bb \|_2^2 + \alpha \| \WW \xx \|_1 \right\}. 
\end{align}
For problems I) and II) we compare all three approaches, but for problem III) we only consider \eqref{method:standard_sparsity} and \eqref{method:weighted_sparsity_standard_fidelity} due to the lack of precise noise information.

\subsection{The inverse heat conduction problem} 
Consider the following initial/boundary-value problem
\begin{equation} \label{D1}
\begin{split}
    &u_t=u_{xx}, \quad x \in (0,\pi), \, t>0, \\
    &u(0,t)=u(\pi,0)=0, \quad t>0, \\
    &u(x,0)=u_0(x), \quad x \in (0, \pi), 
\end{split}
\end{equation}
and its associated forward operator 
\begin{equation*}
    K:u_0 \mapsto u(\cdot,T)|_{(0, \pi/4)}, 
\end{equation*}
where $u_0 \in L^2(0,\pi)$ and $T=0.5$. Note that "data" only is recorded in the subdomain $(0,\pi/4)$ of the entire spatial domain $(0,\pi)$. Hence, the spaces $X$ and $Y$, used in the discussion of \eqref{eq:orig}, are in this example 
\begin{align*}
    X &= L^2(0,\pi), \\
    Y &= L^2(0,\pi/4). 
\end{align*}
Using the technique of separation of variables, the action of $K$ becomes transparent: 
\begin{equation} \label{D2}
    K: \sum_m a_m \sin(m x) \mapsto \left. \sum_m a_m e^{-m^2 T} \sin(m x) \right|_{(0,\pi/4)}. 
\end{equation}

Since the solution $u(x,t)$ of \eqref{D1} is analytic wrt.\ $x$ for any $t>0$, see, e.g., \cite{evans2022partial} and references therein, it follows from the Identity Theorem for analytic functions \cite{krantz2002primer} 
that $K$ is injective. This is the case even though $K$ is a composition of a forward heat conduction operator and a restriction to a subdomain mapping. 

On the other hand, unless $T$ is very close to zero, $e^{-m^2 T} \approx 0$ for $m \geq 4$ because $e^{-4^2} \approx 1.13 \cdot 10^{-7}$. Hence, typically $K (\sin(mx)) \approx 0$ for $m \geq 4$, and the fictitious null space of $K$ will be significant, even with moderate levels of noise in the data. 

We discretized $K$ in the following way: The interval $[\pi/n, \pi-\pi/n]$ was divided into $n=20$ subintervals $I_1, \, I_2, \ldots, \, I_n$ of equal length. The characteristic functions $\phi_1=\chi_{I_1}, \,  \phi_2=\chi_{I_2}, \ldots, \phi_n = \chi_{I_n}$ of these subintervals were used as basis functions yielding a finite dimensional space for the discrete initial condition: 
\begin{equation} \label{eq:K_h_inverse_heat}
    K_h: \sum_{j=1}^n a_j \chi_{I_j} \mapsto u_M(\cdot,T)|_{(0, \pi/4)},  
\end{equation}
where $M=3000$ represents the number of sine modes employed in the truncated sine series approximation of $u(\cdot,T)|_{(0, \pi/4)}$. Note that we only allow the discrete counterpart to $u_0$ to be nonzero in the subinterval $[\pi/n, \pi-\pi/n]$ of $(0, \pi)$ in order to avoid an inconsistency between the initial condition and the homogeneous Dirichlet boundary condition employed in \eqref{D1}. 

The matrix $\AAA$ associated with $K_h$ maps the coefficients $a_1, \, a_2, \, \ldots, \, a_n$ into the function values $u_M(z_1), \, u_M(z_2), \, \ldots, \, u_M(z_n)$, where $z_1, \, , z_2, \, \ldots, \, z_n$ are $n$ equidistant points in the observation interval $(0, \pi/4)$. The action of $K_h$ was implemented in Matlab, using the FFT algorithm to compute the Fourier coefficients to be employed in the truncated version of \eqref{D2}. Furthermore, the weights $w_{k,1}, w_{k,2}, \ldots, w_{k,n}$ associated with the basis functions $\phi_1=\chi_{I_1}, \, \phi_2=\chi_{I_1}, \ldots, \phi_n=\chi_{I_n}$, see \eqref{A6.01}, were stored in a diagonal matrix $\WW$.  

Figure \ref{fig:ex1_weights} contains illustrations of these weights for two different levels of noise. We observe that the weights associated with the observation region $(0,\pi/4) \approx (0,0.785)$ become larger than the weights for the region $(2.5,\pi)$ when the amount of noise increases. This is due to the fact that the size of the fictitious null space increases as the noise increases and the bias caused by the use of a limited observation region. If observations are made throughout the entire interval $(0,\pi)$, then the weights become similar to those shown in Figure \ref{fig:ex1_weights} (a), for both choices of the noise level (illustrations not included).  
\begin{figure}[H]
    \centering
    \begin{subfigure}[b]{0.49\linewidth}        
        \centering
        \includegraphics[width=\linewidth]{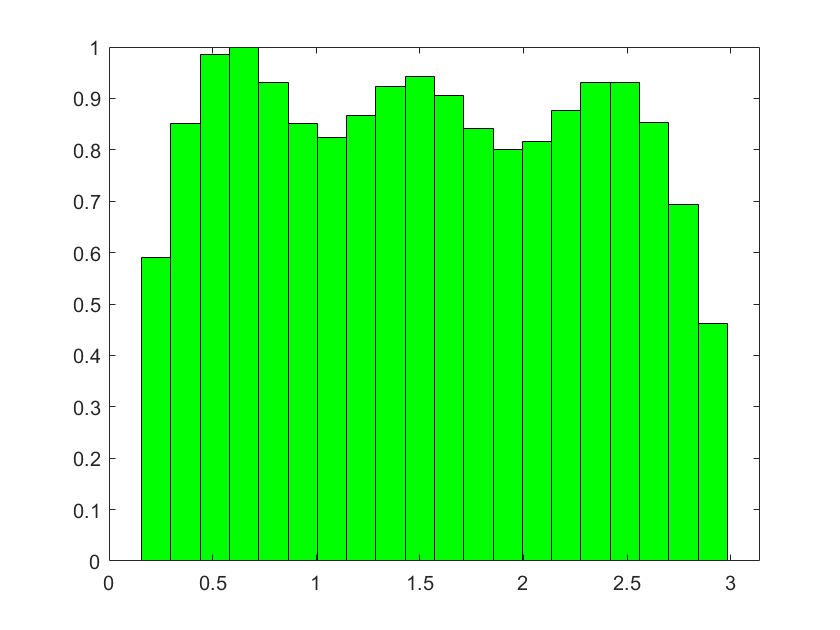}
        \caption{$1 \permil$ noise and $k=3$.}
    \end{subfigure}
    \begin{subfigure}[b]{0.49\linewidth}        
        \centering
        \includegraphics[width=\linewidth]{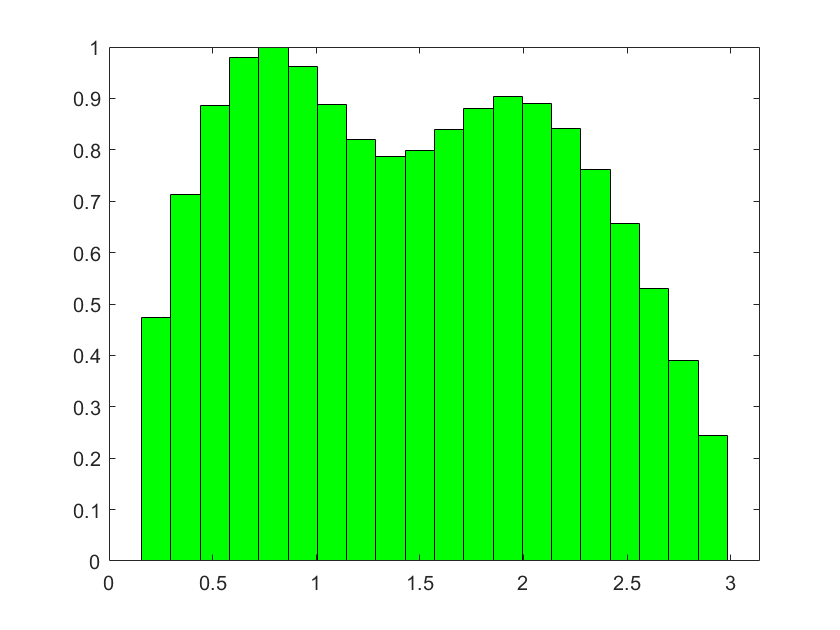}
        \caption{$1 \%$ noise and $k=2$.}
    \end{subfigure}\par
    \caption{The weights \eqref{A6.01} associated with the basis functions for two levels of noise. The criterion \eqref{eq:select_k} was used to determine the size of the truncation parameter $k$.}  
    \label{fig:ex1_weights}
\end{figure}

Figure \ref{fig:ex1_0_001_noise} shows that all the three methods \eqref{method:standard_sparsity}, \eqref{method:weighted_sparsity_modified_fidelity} and \eqref{method:weighted_sparsity_standard_fidelity} yield adequate results in the case of $1 \permil$ noise and when the support of the true initial condition is outside the region $(0,\pi/4)$ where the observations are made. On the other hand, with $1 \%$  noise, see Figure \ref{fig:ex1_0_01_noise}, standard sparsity regularization, panel (b), does not produce inverse solutions of the same quality as the weighted versions, panels (c) and (d), of the methodology. This was also observed with $5 \%$ noise (illustrations not included). As mentioned earlier, we do not have a good theoretical understanding of why there are significant differences, when $\alpha > 0$, between the solutions of \eqref{method:weighted_sparsity_modified_fidelity} and \eqref{method:weighted_sparsity_standard_fidelity}, cf. panels (c) and (d) in Figure \ref{fig:ex1_0_01_noise}. This is an open problem.  

The true initial condition depicted in panel (a) in figures \ref{fig:ex1_0_001_noise} and \ref{fig:ex1_0_01_noise} is not in the span of the basis functions employed to discretize the operator $K$ in \eqref{D2}, see also \eqref{eq:K_h_inverse_heat}. Hence, we did not commit an inverse crime. 

\begin{figure}[H]
    \centering
    \begin{subfigure}[b]{0.49\linewidth}        
        \centering
        \includegraphics[width=\linewidth]{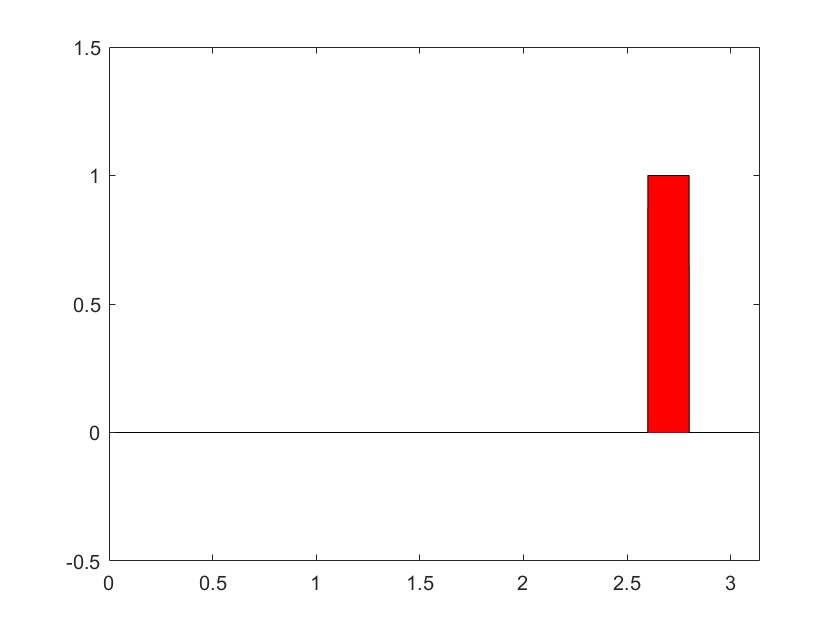}
        \caption{True initial condition.}
    \end{subfigure}
    \begin{subfigure}[b]{0.49\linewidth}        
        \centering
        \includegraphics[width=\linewidth]{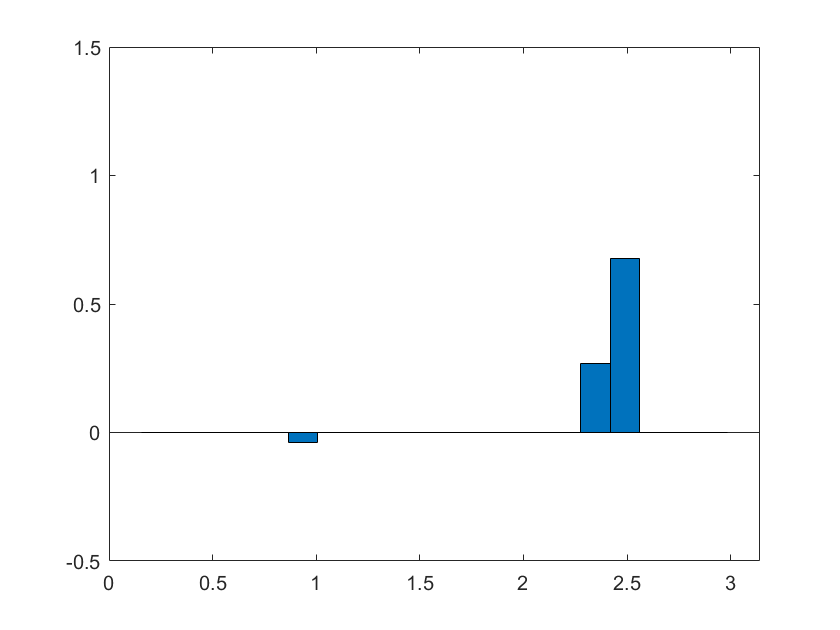}
        \caption{$\min_{\xx} \left\{ \| \AAA \xx - \bb \|_2^2 + \alpha \| \xx \|_1 \right\}$.}
    \end{subfigure}\par
    \begin{subfigure}[b]{0.49\linewidth}        
        \centering
        \includegraphics[width=\linewidth]{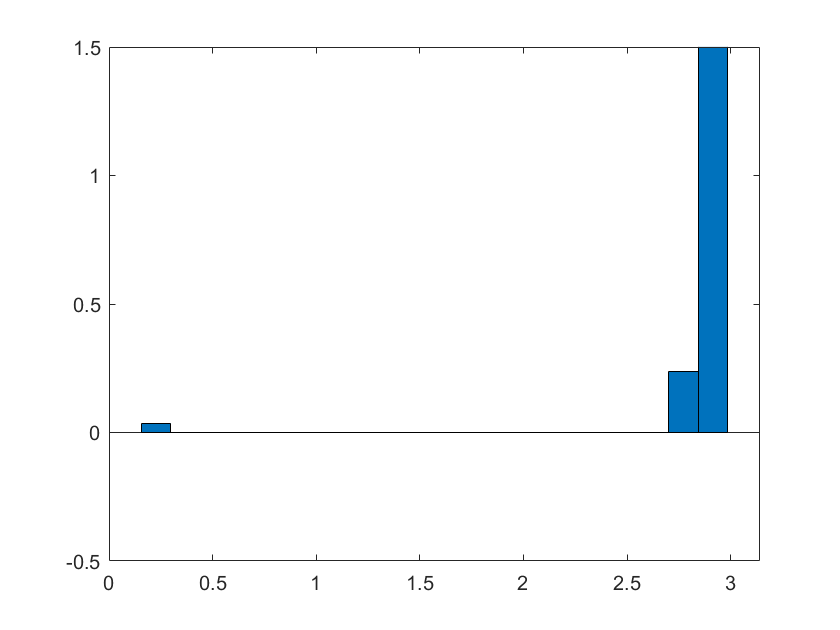}
        \caption{$\min_{\xx} \left\{ \| \AAA^\dagger \AAA \xx - \AAA^\dagger \bb \|_2^2 + \alpha \| \WW \xx \|_1 \right\}$.}
    \end{subfigure}
    \begin{subfigure}[b]{0.49\linewidth}        
        \centering
        \includegraphics[width=\linewidth]{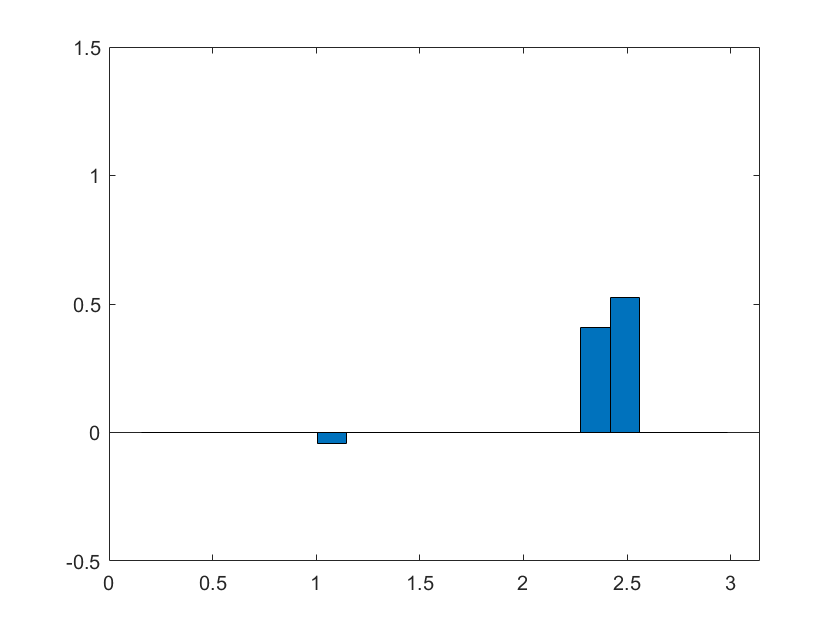}
        \caption{$\min_{\xx} \left\{ \| \AAA \xx - \bb \|_2^2 + \alpha \| \WW \xx \|_1 \right\}$.}
    \end{subfigure}
    \caption{Inverse heat conduction example: $1 \permil$ noise and Morozov's discrepancy principle.}
    \label{fig:ex1_0_001_noise}
\end{figure}

\begin{figure}[H]
    \centering
    \begin{subfigure}[b]{0.49\linewidth}        
        \centering
        \includegraphics[width=\linewidth]{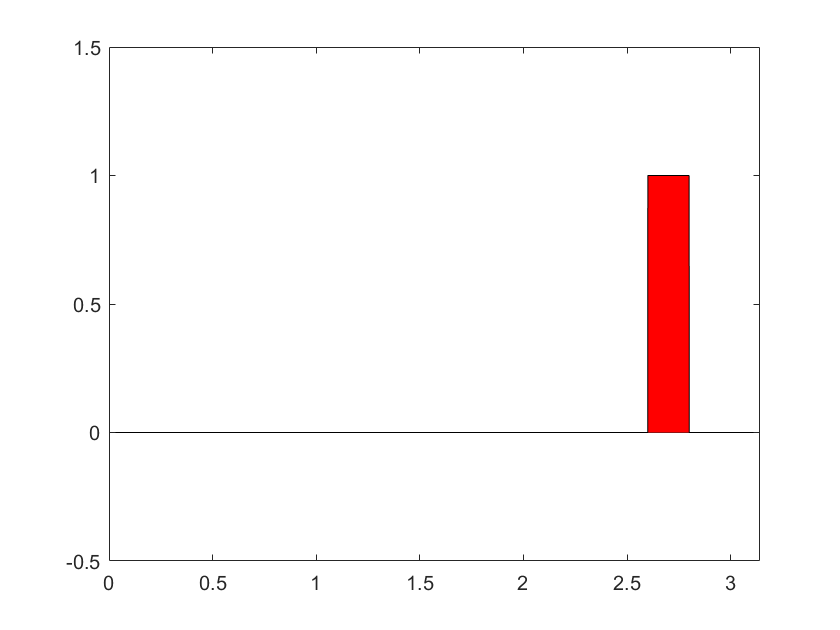}
        \caption{True initial condition.}
    \end{subfigure}
    \begin{subfigure}[b]{0.49\linewidth}        
        \centering
        \includegraphics[width=\linewidth]{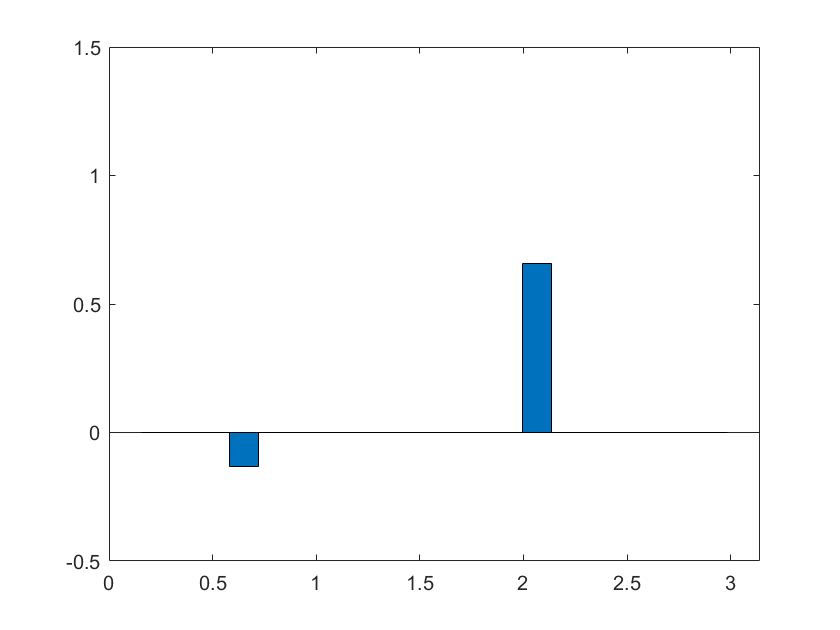}
        \caption{$\min_{\xx} \left\{ \| \AAA \xx - \bb \|_2^2 + \alpha \| \xx \|_1 \right\}$.}
    \end{subfigure}\par
    \begin{subfigure}[b]{0.49\linewidth}        
        \centering
        \includegraphics[width=\linewidth]{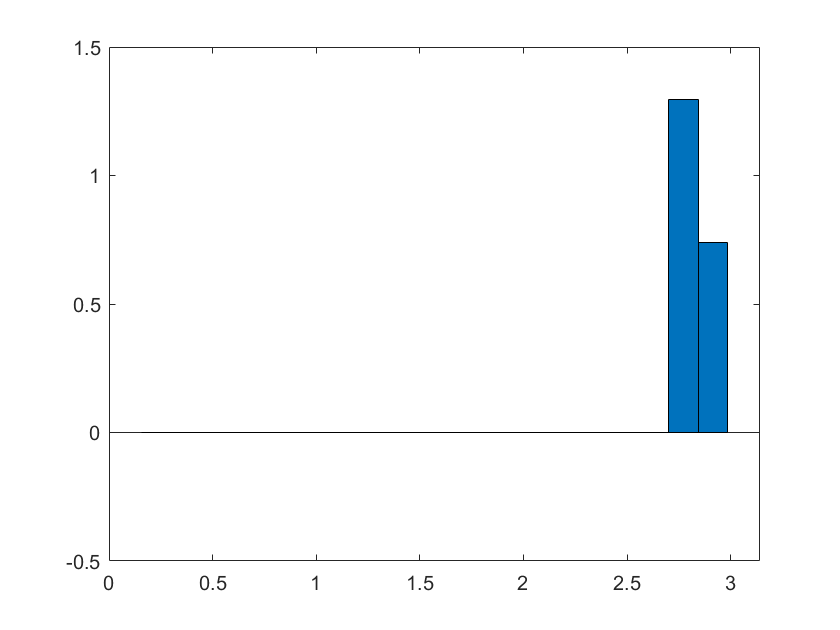}
        \caption{$\min_{\xx} \left\{ \| \AAA^\dagger \AAA \xx - \AAA^\dagger \bb \|_2^2 + \alpha \| \WW \xx \|_1 \right\}$.}
    \end{subfigure}
    \begin{subfigure}[b]{0.49\linewidth}        
        \centering
        \includegraphics[width=\linewidth]{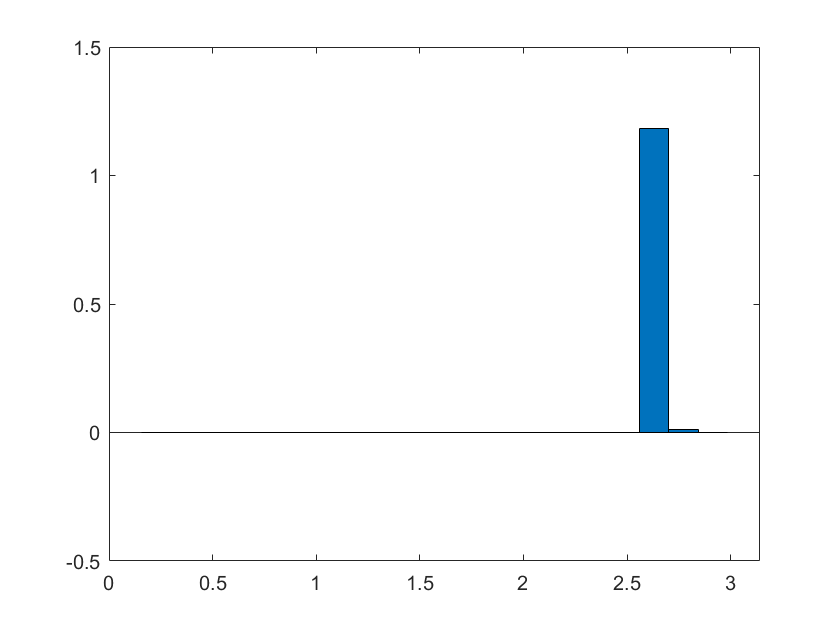}
        \caption{$\min_{\xx} \left\{ \| \AAA \xx - \bb \|_2^2 + \alpha \| \WW \xx \|_1 \right\}$.}
    \end{subfigure}
    \caption{Inverse heat conduction example: $1 \%$ noise and Morozov's discrepancy principle.}
    \label{fig:ex1_0_01_noise}
\end{figure}

\subsection{A Cauchy  problem for Laplace's equation}
Consider the elliptic boundary value problem
\begin{align}\label{eq:exMoon1}
    - \Delta u &= 0, \ \xx \in \Omega, \\ \label{eq:exMoon2}
    u &= f, \ \xx \in \partial\Omegai, \\\label{eq:exMoon3}
    \partial_n u &= 0, \ \xx \in \partial\Omegao,
\end{align}
where $\Omega \subset \mathbb{R}^2$ is an annulus with inner and outer radii $0.5$ and $1$, respectively, see Figure \ref{fig:annulus}. Note that $\partial\Omegai$ and $\partial\Omegao$ represent the inner and outer boundaries of the annulus, and that data for the state $u$ only is available throughout the blue subregion $\Omega_e$. More specifically, the forward operator reads
\begin{equation} \label{eq:forward_Cauchy}
\begin{split}
        Kf: L^2(\partial\Omegai) & \rightarrow L^2(\Omega_e), \\  f &\mapsto u|_{\Omega_e}.
\end{split}
\end{equation}

Since the solution $u$ of \eqref{eq:exMoon1}-\eqref{eq:exMoon3} is harmonic in the interior of $\Omega$, it follows from standard regularity theory that $u$ is analytic and, as in the previous example, we can conclude from the Identity Theorem for analytic functions \cite{krantz2002primer} that $K$ is injective. 
\begin{figure}[H]
    \centering 
    \includegraphics[width=.3\linewidth]{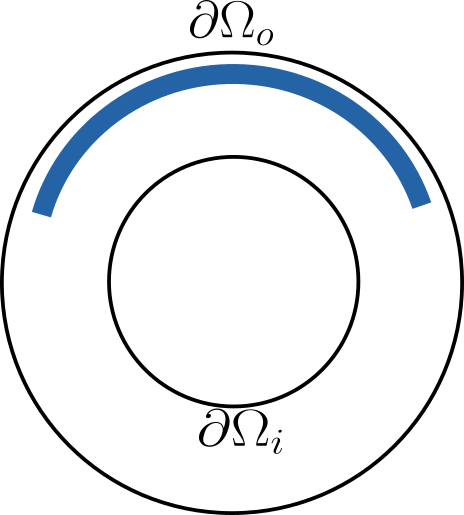}
    \caption{The annulus-shaped domain $\Omega$ for the Cauchy problem. The blue region is the subdomain $\Omega_e$ in which data is recorded.}\label{fig:annulus}
\end{figure}

To run the forward simulations, we discretized $\partial \Omegai$ using $180$ characteristic functions $\chi_j$ associated with equally sized arches, i.e., with arch length $\frac{\pi}{180}$. The forward matrix was thereafter generated by numerically solving \eqref{eq:exMoon1}-\eqref{eq:exMoon3} with $f = \chi_j, j = 1,2,\ldots, 180$. This was done using the FEM software tool FEniCS, where we employed standard piecewise linear Lagrange elements with a mesh for $\Omega$ containing $87,230$ nodes. Finally, the numerical solution $u_h$ was sampled at 120 uniformly distributed points in $\Omega_e$, i.e., the blue region in Figure \ref{fig:annulus}. This procedure resulted in a forward matrix with dimensions $120 \times 180$. 

In the inverse experiments we followed a similar approach, except that $\partial \Omegai$ was discretized using only $120$ characteristic functions, defined in terms of equally sized arches, 
and the FEM simulations were performed on a grid with $64,162$ nodes. This resulted in a quadratic matrix $\AAA$, cf. \eqref{eq:matrix-vector_version}, of size $120 \times 120$. Note that we avoided inverse crimes because different discretizations were used in the forward and inverse computations. 

\begin{figure}[H]
    \centering
    \begin{subfigure}[b]{0.49\linewidth}
        \includegraphics[width=\linewidth]{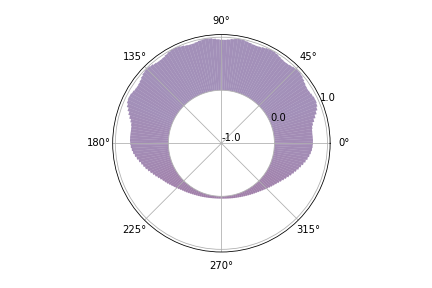}
        \caption{1\% noise}
    \end{subfigure}
    \begin{subfigure}[b]{0.49\linewidth}
        \includegraphics[width=\linewidth]{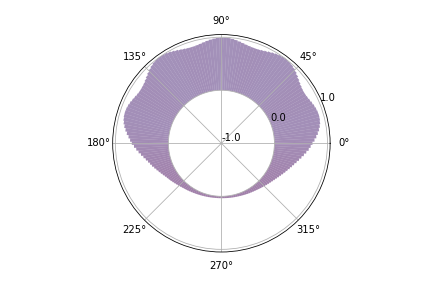}
        \caption{5\% noise}
    \end{subfigure}
        \caption{The weights \eqref{A6.01} associated with the basis functions, employed to discretize the forward operator \eqref{eq:forward_Cauchy}, for two different levels of noise.}
         \label{fig:ex2_weights}
\end{figure}

\begin{figure}[H]
    \centering
    \begin{subfigure}[b]{0.49\linewidth}        
        \centering
        \includegraphics[width=\linewidth]{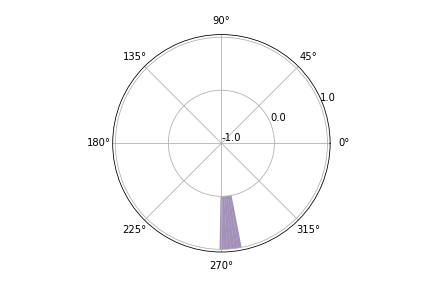}
        \caption{True source.}
    \end{subfigure}
    \begin{subfigure}[b]{0.49\linewidth}        
        \centering
        \includegraphics[width=\linewidth]{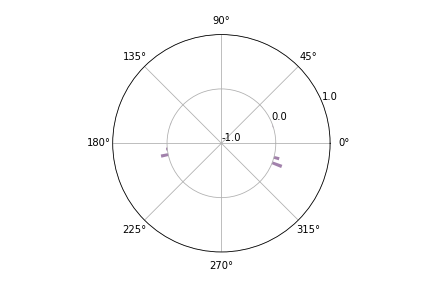}
        \caption{$\min_{\xx} \left\{ \| \AAA \xx - \bb \|_2^2 + \alpha \| \xx \|_1 \right\}$.}
    \end{subfigure}\par
    \begin{subfigure}[b]{0.49\linewidth}        
        \centering
        \includegraphics[width=\linewidth]{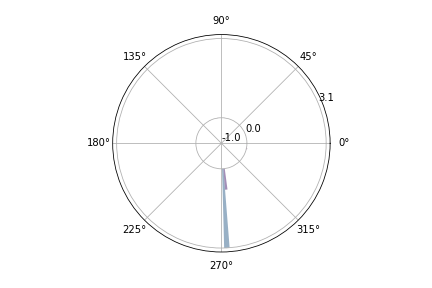}
        \caption{$\min_{\xx} \left\{ \| \AAA^\dagger \AAA \xx - \AAA^\dagger \bb \|_2^2 + \alpha \| \WW \xx \|_1 \right\}$.}
    \end{subfigure}
    \begin{subfigure}[b]{0.49\linewidth}        
        \centering
        \includegraphics[width=\linewidth]{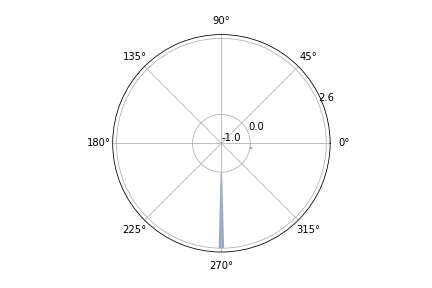}
        \caption{$\min_{\xx} \left\{ \| \AAA \xx - \bb \|_2^2 + \alpha \| \WW \xx \|_1 \right\}$.}
    \end{subfigure}
    \caption{Inverse Cauchy problem: $1 \%$ noise and Morozov's discrepancy principle.}
    \label{fig:ex3_0_01_noise}
\end{figure}

We ran several experiments with true sources located in different subregions of $\partial \Omegai$. For sources located in the upper half-plane, there was no qualitative difference between the regularized solutions with and without weighting, cf. \eqref{method:standard_sparsity}-\eqref{method:weighted_sparsity_standard_fidelity}, and we decided not to present any of these results. All the three methods performed rather well. 

However, when the true source was located far away from where measurements are made, the weighting had a significant impact: In panels (c) and (d) of Figure \ref{fig:ex3_0_01_noise} we observe that the weighted solutions succeeded in recovering the position of the true source, whereas panel (b) reveals that the standard approach did not handle this case very well. 
All figures in this example display polar plots, and the purple bars show the angle/position of the source on the inner boundary of the annulus, cf. Figure \ref{fig:annulus}. The number printed between 0$^\circ$ and 45$^\circ$ in each polar plot represents the magnitude of the source.

Even though the weighted \rem{sparsity} procedures managed to recover the correct position of the source, the magnitude is too large and the spread is too small, see Figure \ref{fig:ex3_0_01_noise}. In Figure \ref{fig:ex3_0_01_noiseINEQ} we illustrate that this might be rectified by introducing an inequality constraint on the control. However, this is only possible if a qualified choice for the inequality bound is available.

\begin{figure}[H]
    \centering
    \begin{subfigure}[b]{0.49\linewidth}        
        \centering
        \includegraphics[width=\linewidth]{moon/posterior/truebc_0.01_noise.png}
        \caption{True source.}
    \end{subfigure}
    \begin{subfigure}[b]{0.49\linewidth}        
        \centering
        \includegraphics[width=\linewidth]{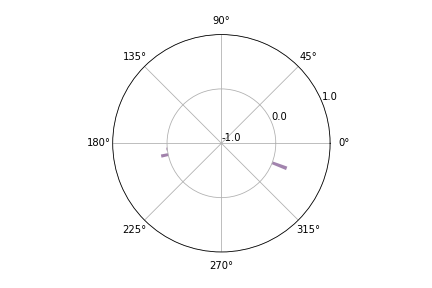}
        \caption{$\min_{0\leq\xx\leq1} \left\{ \| \AAA \xx - \bb \|_2^2 + \alpha \| \xx \|_1 \right\}$.}
    \end{subfigure}\par
    \begin{subfigure}[b]{0.49\linewidth}        
        \centering
        \includegraphics[width=\linewidth]{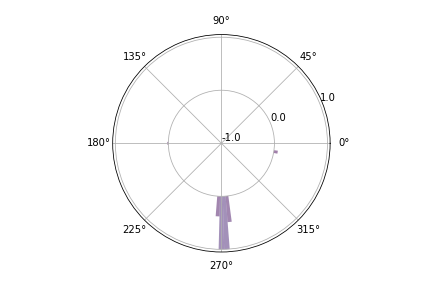}
        \caption{$\min_{0\leq\xx\leq1} \left\{ \| \AAA^\dagger \AAA \xx - \AAA^\dagger \bb \|_2^2 + \alpha \| \WW \xx \|_1 \right\}$.}
    \end{subfigure}
    \begin{subfigure}[b]{0.49\linewidth}        
        \centering
        \includegraphics[width=\linewidth]{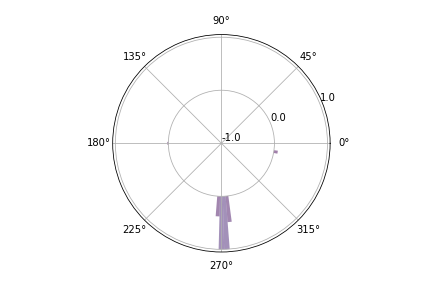}
        \caption{$\min_{0\leq\xx\leq1} \left\{ \| \AAA \xx - \bb \|_2^2 + \alpha \| \WW \xx \|_1 \right\}$.}
    \end{subfigure}
    \caption{Inverse Cauchy problem: $1 \%$ noise and Morozov's discrepancy principle.}
    \label{fig:ex3_0_01_noiseINEQ}
\end{figure}

\begin{figure}[H]
    \centering
    \begin{subfigure}[b]{0.49\linewidth}        
        \centering
        \includegraphics[width=\linewidth]{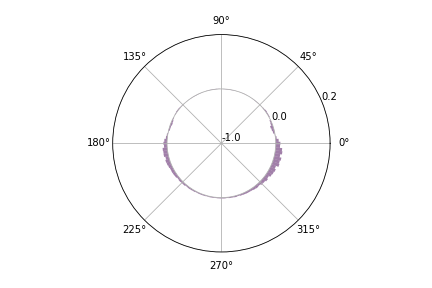}
        \caption{\rem{$\min_{\xx} \left\{ \| \AAA \xx - \bb \|_2^2 + \alpha \| \xx \|_2^2 \right\}$.}}
    \end{subfigure}
    \begin{subfigure}[b]{0.49\linewidth}        
        \centering
        \includegraphics[width=\linewidth]{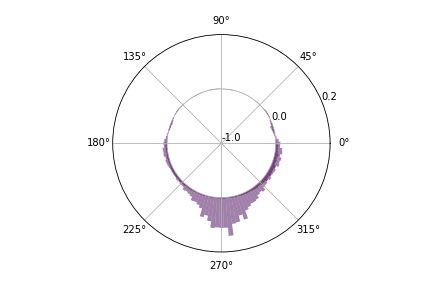}
        \caption{\rem{$\min_{\xx} \left\{ \| \AAA \xx - \bb \|_2^2 + \alpha \| \WW\xx \|_2^2 \right\}$.}}
    \end{subfigure}
    \caption{\rem{Inverse Cauchy problem, employing standard and weighted Tikhonov regularization: $1 \%$ noise and Morozov's discrepancy principle. The true source is shown in Figure \ref{fig:ex3_0_01_noise}(a).}}
    \label{fig:ex3_tikh}
\end{figure}

\rem{In contrast to sparsity recovery, using weighted quadratic Tikhonov regularization produces a solution where the magnitude is too small and the spread is too wide, cf. Figure \ref{fig:ex3_tikh}. Nevertheless, the maximum of the reconstruction shown in panel (b) is correctly located. These computations suggest that a weighted version of the method proposed in \cite{2009_Lu} could produce results of the same sparse quality as those obtained with \eqref{A6}.}

Recall the definitions \eqref{eq:tfwd} and \eqref{eq:TSVD} of the truncated operators $\Kk$ and $\Kk^\dagger$. 
To illustrate why the weights, which are displayed in Figure \ref{fig:ex2_weights}, has a significant impact on the present source recovery problem, we have included plots of the $9$ first  singular vectors $\mathbf{v}_1, \mathbf{v}_2, \ldots ,\mathbf{v}_9$ of $\AAA$. Figure \ref{fig:ex3_svd} shows that $\mathbf{v}_1$ is almost zero in all indices corresponding to nodes located in the lower half-plane of the annulus, and that the magnitudes associated with these positions in $\mathbf{v}_k$ only increases slowly wrt.\ $k$.

\begin{figure}[H]
    \centering
    \begin{subfigure}[b]{0.32\linewidth}        
        \centering
        \includegraphics[width=\linewidth]{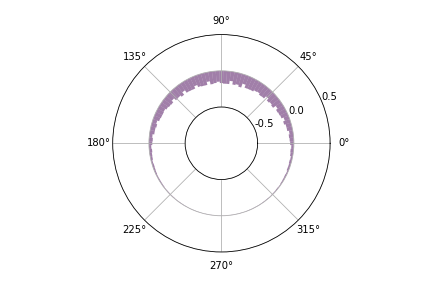}
        \caption{$\mathbf{v}_1$}
    \end{subfigure}
    \begin{subfigure}[b]{0.32\linewidth}        
        \centering
        \includegraphics[width=\linewidth]{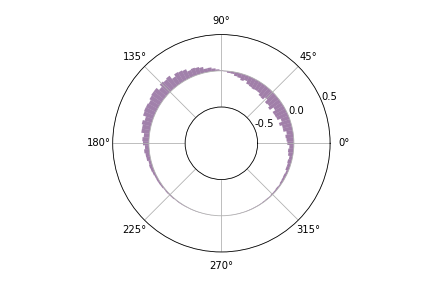}
        \caption{$\mathbf{v}_2$}
    \end{subfigure}
    \begin{subfigure}[b]{0.32\linewidth}        
        \centering
        \includegraphics[width=\linewidth]{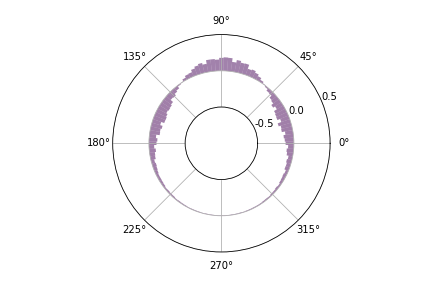}
        \caption{$\mathbf{v}_3$}
    \end{subfigure}\par
    \begin{subfigure}[b]{0.32\linewidth}        
        \centering
        \includegraphics[width=\linewidth]{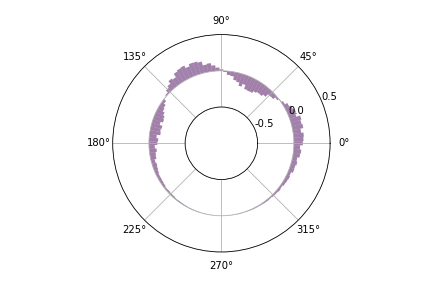}
        \caption{$\mathbf{v}_4$}
    \end{subfigure}
    \begin{subfigure}[b]{0.32\linewidth}        
        \centering
        \includegraphics[width=\linewidth]{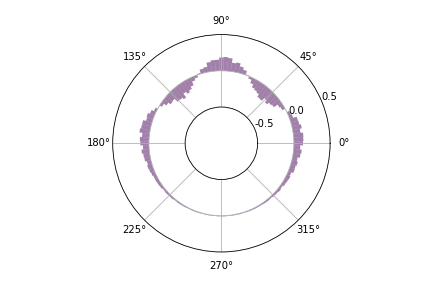}
        \caption{$\mathbf{v}_5$}
    \end{subfigure}
    \begin{subfigure}[b]{0.32\linewidth}        
        \centering
        \includegraphics[width=\linewidth]{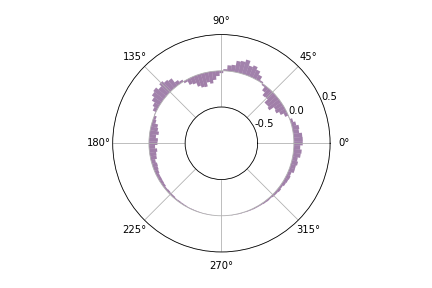}
        \caption{$\mathbf{v}_6$}
    \end{subfigure}\par
        \begin{subfigure}[b]{0.32\linewidth}        
        \centering
        \includegraphics[width=\linewidth]{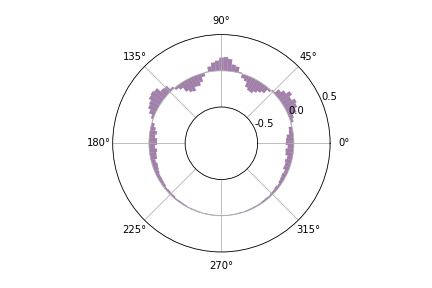}
        \caption{$\mathbf{v}_7$}
    \end{subfigure}
    \begin{subfigure}[b]{0.32\linewidth}        
        \centering
        \includegraphics[width=\linewidth]{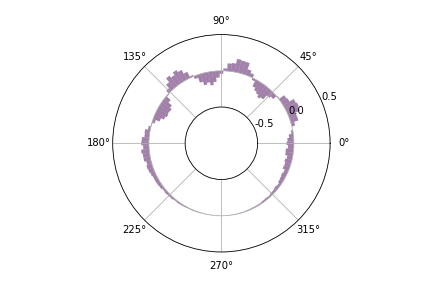}
        \caption{$\mathbf{v}_8$}
    \end{subfigure}
    \begin{subfigure}[b]{0.32\linewidth}        
        \centering
        \includegraphics[width=\linewidth]{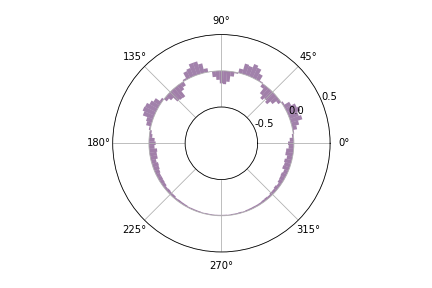}
        \caption{$\mathbf{v}_9$}
    \end{subfigure}\par
    \caption{First nine right singular vectors.}
    \label{fig:ex3_svd}
\end{figure}

\subsection{Electrical Impedance Tomography}
Electrical Impedance Tomography (EIT) is an imaging technology that aims at reconstructing a body's interior electrical conductivity distribution from electrode measurements of currents and voltages on the surface of the body. For the presented case we rely on the code and data for the Kuopio Tomography Challenge 2023 (KTC23) \cite{Mikko}.

The most accurate mathematical model for EIT is the complete electrode model (CEM) \cite{somersalo1992}. The body is modeled by a bounded domain $\Omega \subset \mathbb{R}^2$ with smooth boundary $\partial \Omega$ and conductivity $\sigma.$ For consistency with KTC23, we fix $\Omega = B(0,R),\, R = 11.5$cm. On the boundary of the domain, $M=32$ electrodes of equal size are mounted equidistantly, and through each electrode $E_m$, the input current $I_m$ is controlled giving the input current pattern $I = [I_1,I_2, \ldots, I_M]^T \in \mathbb R_\diamond^M$, where $\mathbb R_\diamond^M$ denotes the subspace of vectors in $\RM$ with components summing to zero. $I$ generates an interior electric potential $u\in H^1(\Omega)$ that together with the boundary voltages $U=(U_1,U_2,\ldots,U_M) \in \mathbb R_\diamond^M,$ constant on each electrode $E_m,$ solves
\begin{align}
    -\nabla \cdot (\sigma \nabla u)&=0, &\mathrm{in}&\quad\Omega, \label{eq:CEM1}\\
    u+z\sigma\frac{\partial u}{\partial \nu}&=U_m, &\mathrm{on}& \quad E_m,\, m= 1,\ldots, M, \\
    \int_{E_m} \nu\cdot(\sigma \nabla u)dS &= I_m, &\mathrm{on}& \quad E_m,\, m= 1,\ldots, M, \\
    \frac{\partial u}{\partial \nu} &=0, & \mathrm{on}& \quad\partial\Omega\setminus \bigcup_{m=1}^M E_m,\label{eq:CEM4}
    \end{align}
where $\nu$ denotes the outward pointing unit normal vector at the boundary $\partial\Omega$. The factor $z$ is the so-called contact impedance. Note that the solution to \eqref{eq:CEM1}-\eqref{eq:CEM4} is the pair $(u,U) \in H^1(\Omega) \oplus R_\diamond^M,$ and that  $U=\in \mathbb R_\diamond^M$  implements a zero sum condition thereby grounding the voltage potential; this gives uniqueness for the problem. \rem{Outside the electrodes $\{ E_m \}$ the homogeneous Neumann boundary condition \eqref{eq:CEM4} is put; this reflects the physical assumption that the object under consideration is electrically insulated. This is in contrast to deblurring tasks, for which appropriate boundary conditions can be chosen in order to promote certain computational and recovery properties \cite{donatelli2004anti,ng1999fast}.}

In EIT several current patterns $\mathbf{I} = \{I^{n}\},\; n = 1,2,\ldots N,$ are injected sequentially. The corresponding boundary voltages are collected in $\mathbf{U}=\{U^{n}\};$ $N$ denotes the number of experiments. In line with \cite{Mikko}, we consider voltage differences between consecutive electrodes and collect them in the matrix $\mathbf{V.}$ That is, the $i$'th element of the $n$'th column of $\mathbf{V}$ is $U^n_{i+1}-U^n_{i}.$

The forward operator is defined by $F \colon \sigma \mapsto F(\sigma)=\mathbf V.$ The inverse problem 
 of EIT then provides measured data $\mathbf{V}_\text{meas}$ and asks for a solution $\sigma$ to 
\begin{align*}
    F(\sigma) = \mathbf{V}_\text{meas};
\end{align*}
this is a non-linear and highly ill-posed problem. 

A first approach is often to linearize $F$ around a constant base conductivity $\sigma_0,$ i.e.,\ 
\begin{align*}
        F(\sigma)-F(\sigma_0) \approx J(\sigma_0) h 
\end{align*}
using  the notation $J(\sigma_0)$ for the derivative (Jacobian) of $F$ and $h = \sigma-\sigma_0$ for the difference. The linearized problem is then given by the linear equation
\begin{align}\label{linEIT}
         J(\sigma_0) h  = \mathbf{V}_\text{meas} - F(\sigma_0).
\end{align} 
In \cite{Mikko} measurements corresponding to an empty tank is provided, and hence $F(\sigma_0)$ can be replaced by those measurements in the above equation. When doing so, systematic modeling errors can be reduced.

In the case of infinite and continuous boundary measurements, the reconstruction problem is known as the Calder\'on problem in recognition of \cite{Calderon1980}. The linearized Calder\'on problem can be cast in $X=L^2(\Omega)$ (under mild conditions) and is uniquely solvable. We will therefore in the sequel employ the approach \eqref{linEIT} in the spirit of the methods developed throughout this paper.

In accordance with the methods presented in \cite{Mikko}, we implement \eqref{linEIT} using Finite Elements with $1602$ basis functions $\phi_j$ thus arriving at a matrix equation \eqref{eq:matrix-vector_version} as the finite dimensional representations of \eqref{linEIT}. Recall that $h$ and - after discretization - $\xx$ represents the conductivity relative to the base/background conductivity. To control the magnitude of this conductivity, we solve \eqref{method:standard_sparsity} and  \eqref{method:weighted_sparsity_standard_fidelity} with box constraints on $\xx$, and obtain the reconstructions seen in Figure~\ref{fig:eit}. We clearly see a vast improvement from the weighting. In Table~\ref{tab:ssim} we display the Structural Similarity Index Measure (SSIM) that was used in the contest to quantify the reconstructions. We note that our score is on par with  competing methods, e.g.,\ from \cite{Amal}.

\begin{table}[H]
\begin{center}
\begin{tabular}{||c | c c||} 
 \hline
 \textbf{Phantom (row)} & \textbf{Unweighted} & \textbf{Weighted} \\ [0.5ex] 
 \hline
 1 & 0.131 & 0.856 \\ 
 \hline
 2 & 0.033 & 0.881 \\
 \hline
 3 & 0.534 & 0.930 \\
 \hline
 4 & 0.504 & 0.850 \\
 \hline
\end{tabular}
\caption{The SSIM scoring of the results shown in Figure \ref{fig:eit}.}
\label{tab:ssim}
\end{center}
\end{table}

\begin{figure}[H]
    \centering
    \begin{subfigure}[b]{0.32\linewidth}        
        \centering 
        \textbf{True inclusions\\ }
    \end{subfigure}
    \begin{subfigure}[b]{0.32\linewidth}        
        \centering 
        {\textbf{Unweighted\\ $\left\{ \| \AAA \xx - \bb \|_2^2 + \alpha \| \xx \|_1 \right\}$.}}
    \end{subfigure}
    \begin{subfigure}[b]{0.32\linewidth}        
        \centering 
        {\textbf{Weighted\\ $\left\{ \| \AAA \xx - \bb \|_2^2 + \alpha \| \WW \xx \|_1 \right\}$.}}
    \end{subfigure}\par
    \begin{subfigure}[b]{0.32\linewidth}        
        \centering 
        \includegraphics[width=\linewidth]{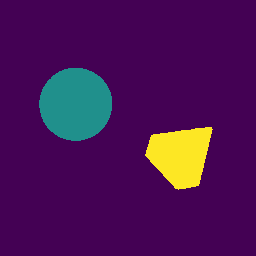}
    \end{subfigure}
    \begin{subfigure}[b]{0.32\linewidth}        
        \centering
        \includegraphics[width=\linewidth]{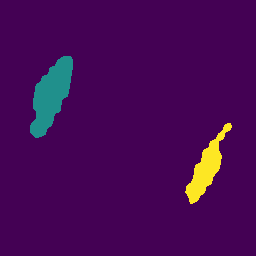}
    \end{subfigure}
    \begin{subfigure}[b]{0.32\linewidth}        
        \centering
        \includegraphics[width=\linewidth]{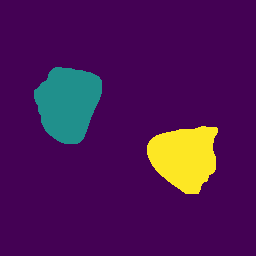}
    \end{subfigure}\par
    \centering
    \begin{subfigure}[b]{0.32\linewidth}        
        \centering
        \includegraphics[width=\linewidth]{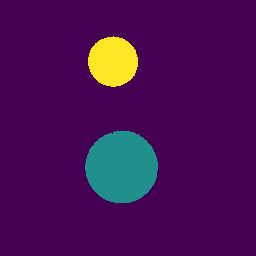}
    \end{subfigure}
    \begin{subfigure}[b]{0.32\linewidth}        
        \centering
        \includegraphics[width=\linewidth]{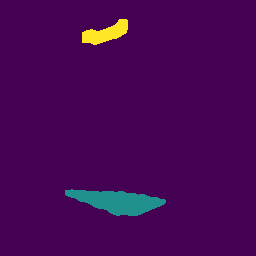}
    \end{subfigure}
    \begin{subfigure}[b]{0.32\linewidth}        
        \centering
        \includegraphics[width=\linewidth]{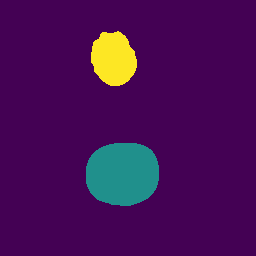}
    \end{subfigure}\par
    \centering
    \begin{subfigure}[b]{0.32\linewidth}        
        \centering
        \includegraphics[width=\linewidth]{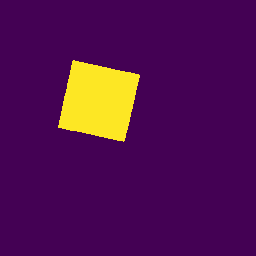}
    \end{subfigure}
    \begin{subfigure}[b]{0.32\linewidth}        
        \centering
        \includegraphics[width=\linewidth]{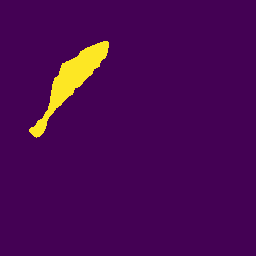}
    \end{subfigure}
    \begin{subfigure}[b]{0.32\linewidth}        
        \centering
        \includegraphics[width=\linewidth]{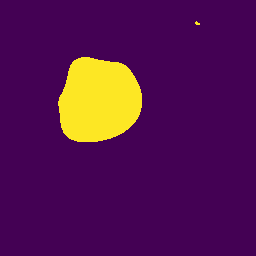}
    \end{subfigure}\par
        \centering
    \begin{subfigure}[b]{0.32\linewidth}        
        \centering
        \includegraphics[width=\linewidth]{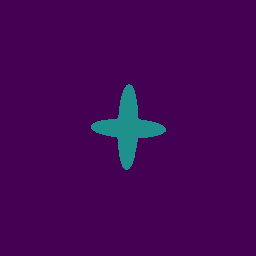}
    \end{subfigure}
    \begin{subfigure}[b]{0.32\linewidth}        
        \centering
        \includegraphics[width=\linewidth]{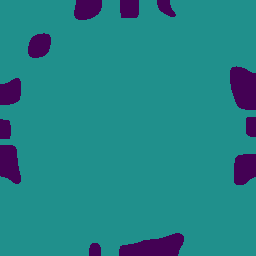}
    \end{subfigure}
    \begin{subfigure}[b]{0.32\linewidth}        
        \centering
        \includegraphics[width=\linewidth]{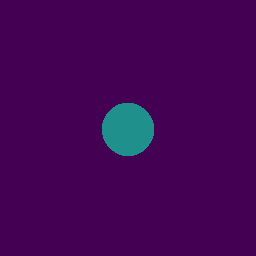}
    \end{subfigure}\par
    \caption{Comparison of the true inclusions and the reconstructions obtained by solving the linearized EIT model, using sparsity regularization and box constraints. The second and third columns contain the results obtained with the unweighted/standard and weighted schemes, respectively.}
    \label{fig:eit}
\end{figure}

\section{Conclusions}
In this work we have demonstrated that, for inverse problems with injective forward operators, the introduction of a fictitious null space and a corresponding weighting can counteract an unwanted bias; we have focused on sparsity regularization, but similar results hold for standard Tikhonov regularization.
Our computational examples document that, for some PDE-driven inverse problems, the strategy is successful, even for a problem arising in EIT with real data.

A few concluding remarks are appropriate: The choice of the fictitious null space inherently depends on the data noise, and the defining truncation parameter $k$ can, as in other regularization methods, be chosen by different methods (e.g., Morozov's discrepancy principle). The optimal and practical choices of both $k$ and the regularization $\alpha$ are left for further studies.

Our theoretical results show that, with the proper weighting, a source given by a single basis function can be identified up to scaling. This is a sanity check; a similar result for a more complex source is desired. Our computational results show that such a result is indeed plausible.

A final remark concerns the data discrepancy/fidelity term. We set out with a standard least squares term, but some of the recovery results are established for a problem which involves an approximation $\Kk^\dagger$ of the pseudo inverse of the forward operator. This can, unless the truncation parameter $k$ is chosen in accordance with the noise level, lead to severe error amplification due to the presence of very small singular values. 
However, our numerical results indicate that the use of $\Kk^\dagger$ in the fidelity term yields better reconstructions than standard methods. Again we leave the theoretical justification for future work.



\section*{Acknowledgments}
Kim Knudsen was supported by the Villum Foundation (grant no. 25893). \rem{We would also like to thank the anonymous referees for their valuable suggestions and comments.}

\appendix

\section{\rem{Proof of Lemma \ref{prop:maximum}}} \label{sec:proof_maximum}
    \rem{From the proof of Proposition \ref{prop:bp}, we have that 
    $$S = \{x \in X: \Kk x = \Kk\phi_j\},$$ cf. the reasoning leading to \eqref{A11}.  
    This implies that
    $$\bar{x}_k = \Kk^\dagger \Kk \phi_j = \Pk\phi_j,$$ see \eqref{eq:Pk}. 
    Hence, by expanding $\bar{x}_k$ in the $\{\phi_i\}$-basis, and invoking the definition \eqref{eq:W} of $\Wk$, we get
    \begin{align}
        \nonumber
        \Wk^{-1}\bar{x}_k &= \Wk^{-1} \sum_i (\bar{x}_k,\phi_i) \phi_i \\
        \nonumber
        &= \sum_i \wki^{-1}(\bar{x}_k,\phi_i) \phi_i \\
        \nonumber
        &= \sum_i \wki^{-1}(P_k\phi_j,\phi_i) \phi_i \\
        \nonumber
        &= \sum_i \wki^{-1}(P_k\phi_j,P_k\phi_i) \phi_i \\
        \label{eq:maximumIndexExpansion}
        &= \wkj \sum_i \left(\frac{\Pk\phi_j}{\wkj}, \frac{\Pk\phi_i}{\wki}\right) \phi_i, 
    \end{align}
    where the second last equality follows from the fact that $\Pk$ is an orthonormal projection. That is, 
    \begin{equation} \label{eq:maximumIndex}
        [\Wk^{-1}\bar{x}_k]_i = \wkj \left(\frac{\Pk\phi_j}{\wkj}, \frac{\Pk\phi_i}{\wki}\right). 
    \end{equation}}
    
    \rem{Recall the definition \eqref{A6.01} of the weights $\{ \wki \}$. The assumption that $\|\Pk\phi_j\| \geq \tau$ implies that $\wkj = \|\Pk\phi_j\|$ and, since $\wki \geq \|\Pk\phi_i\|$, it follows from \eqref{eq:maximumIndex} that $j$ is in the $\argmax$ set \eqref{eq:minnorm}. Moreover, the assumption \eqref{eq:nonpar} yields that  
    \begin{equation*}
    \Pk \phi_l \neq c \Pk \phi_q, \quad l \neq q, c \in \mathbb{R}, 
    \end{equation*}
    and hence the $\argmax$-set \eqref{eq:minnorm} can only contain $j$ because applying the Cauchy-Schwarz inequality to the inner product on right-hand-side of \eqref{eq:maximumIndex} yields a strict inequality.} 

\bibliographystyle{plain}
\bibliography{references}

\end{document}